\documentclass[11pt,a4paper,reqno]{amsart}

\usepackage[utf8]{inputenc}

\usepackage{array,amsbsy,amscd,amsfonts,color,amssymb,amstext,
amsmath,latexsym,amsthm,amscd,multicol,graphicx,tikz}
\usepackage[english]{babel}
\usepackage{hyperref}

\numberwithin{equation}{section}

\makeindex

\newtheorem{theorem}{Theorem}
\newtheorem{proposition}[theorem]{Proposition}
\newtheorem{lemma}[theorem]{Lemma}

\theoremstyle{definition}
\theoremstyle{definition}\newtheorem{definition}[theorem]{Definition}
\theoremstyle{definition}
\theoremstyle{definition}\newtheorem{example}[theorem]{Example}
\theoremstyle{definition}
\theoremstyle{definition}
\theoremstyle{definition}

\numberwithin{theorem}{section}

\def\proofof [#1] {\noindent {\bf Proof of #1. } }

\def\v #1.{\mathord{\raise 3pt\hbox{\mathsurround=0pt $\mathop\vee\limits^{#1}$\mathsurround=5pt}}}

\def\al #1.{{\mathcal{#1}}}

\renewcommand{\AA}{\mathfrak{A}}

\newcommand{\A}{\mathcal{A}}

\newcommand{\I}{\mathcal{I}}

\renewcommand{\H}{\mathcal{H}}

\newcommand{\N}{\mathbb{N}}
\newcommand{\NN}{\mathbb{N}_0}
\newcommand{\R}{\mathbb{R}}
\newcommand{\C}{\mathbb{C}}

\newcommand{\T}{\mathcal{T}}

\newcommand{\unit}{\mathbf{1}}

\newcommand{\bp}{\begin{proof}}
\newcommand{\ep}{\end{proof}}
\newcommand{\bdp}{\begin{dproof}}
\newcommand{\edp}{\end{dproof}}
\newcommand{\ra}{\rightarrow}

\newcommand{\SVir}{\operatorname{SVir}}

\newcommand{\Ad}{\operatorname{Ad }}
\newcommand{\ad}{\operatorname{ad }}

\newcommand{\rmd}{\operatorname{d}}
\newcommand{\rmi}{\operatorname{i}}

\newcommand{\acont}{\operatorname{anal.cont}}

\newcommand{\rme}{\operatorname{e}}

\newcommand{\sgn}{\operatorname{sgn }}

\newcommand{\dom}{\operatorname{dom }}
\newcommand{\im}{\operatorname{im }}

\newcommand{\Cci}{C^{\infty}}

\newcommand{\ie}{{i.e.,\/}\ }

\newcommand{\eg}{{e.g.\/}\ }
\newcommand{\cf}{{cf.\/}\ }

\author{Robin Hillier} 
\thanks{Formerly supported as Marie-Curie Fellow of the Istituto Nazionale di
Alta Matematica, Roma, and by the ERC
Advanced Grant 227458 ``Operator Algebras and Conformal Field Theory".}
\title{Local-Entire Cyclic Cocycles for \linebreak Graded Quantum Field Nets}
\address{Department of Mathematics and Statistics, Lancaster University,
Lancaster LA1 4YF, UK\linebreak
E-mail: {\tt r.hillier@lancaster.ac.uk}}

\begin{document}

\begin{abstract}
In a recent paper we studied general properties of super-KMS functionals on
graded quantum dynamical systems coming from graded translation-covariant
quantum field nets over $\R$, and we carried out a detailed analysis of these
objects on certain models of superconformal nets. In the present article we show
that these locally bounded functionals give rise to local-entire cyclic cocycles
(generalized JLO cocycles) which are homotopy-invariant for a suitable class of
perturbations of the dynamical system. Thus we can associate meaningful
noncommutative geometric invariants to those graded quantum dynamical systems.
\end{abstract}

\maketitle


\section{Introduction}

KMS states on C*-algebras play a crucial role in quantum statistical mechanics
and operator algebras, providing among other things a meaningful abstraction of
thermodynamical equilibrium states on quantum systems, \cf \cite{MS,BR2}. They
have been considered in the framework of algebraic quantum field theory in
several places like \cite{BJ,Haag} and more recently in \cite{CLTW1,CLTW2}, and
that is also the context we are interested in here.

Algebraic quantum field theory has been developed as an operator algebraic
mathematically rigorous approach to quantum field theory using nets of operator
algebras \cite{Haag}. One of
its many important aspects is supersymmetry, an internal symmetry between
bosons and fermions, \ie even and odd elements of the algebra, respectively.
Although physical experimental confirmation is lacking so far, its mathematical
structure is very rich and extends to the general context of C*-algebras, in
which we are going to work here. A famous application of supersymmetry is found
in Connes's concept of spectral triples \cite{Con94}. Given such a spectral
triple with $\theta$-summability conditions, one constructs in a natural way a
``super-Gibbs functional", \ie a supersymmetric or graded version of the usual
Gibbs states (special cases of KMS states) from statistical mechanics. This
super-Gibbs functional then gives rise to an entire cyclic cocycle, \cf
also \cite{Con89,JLO1}; it turns out to be a noncommutative geometric invariant
for certain  ``regular perturbations" of the spectral triple and its
corresponding dynamics. This construction
is the starting point for a noncommutative geometric description of
graded-local 
conformal nets of quantum fields over the circle $S^1$, as achieved in
\cite{CHL12}. The cocycles there give rise to geometric invariants, which, in
particular, recover parts of the
representation theory of the graded-local conformal net, as partly already
suggested in \cite{Lo01,KL06,CKL,CHKL}.

It seems natural to ask whether this construction of entire cyclic cocycles can
be generalized from spectral triples and super-Gibbs functionals to more
general dynamics (in particular translations), supersymmetry and functionals,
and whether it still gives
rise to noncommutative geometric invariants and thermodynamical interpretations.
In fact, Jaffe, Lesniewski and Wisniowski \cite{JLW} and independently Kastler
\cite{Kas} took the first steps into that direction.
They introduced a graded version of the KMS condition for states; we call it
the super-KMS condition and the functionals satisfying this condition
we call super-KMS functionals. They showed that bounded super-KMS
functionals give rise to entire cyclic cocycles. Unfortunately, at that time
several no-go theorems were still unknown, in particular that nontrivial
super-KMS functionals for translation-covariant quantum field nets cannot
be bounded \cite{BL00}, which turned their constructions out to be inapplicable
here. It took several years to construct at least one first example of algebraic
supersymmetry with unbounded but locally bounded super-KMS functionals and
associated local-entire cyclic cocycle for the supersymmetric free field
\cite{BG}, which, however, still had to be put into the framework of graded nets
of von Neumann algebras. In the preceding paper \cite{Hil1}, we took precisely
that step: having studied a few general aspects of super-KMS
functionals, we carried out a detailed analysis of super-KMS functionals for 
the supersymmetric free field net and subsequently also for some other models.
The meaning of those functionals relates to supersymmetric dynamics and phase
transitions as briefly outlined there. Yet we have not treated relations to
entire cyclic cohomology \cite{JLW} so far, which was actually one of the
initial motivations. 

Our question is thus: \emph{Do super-KMS functionals for graded
translation-covariant nets give rise to entire cyclic cocycles and geometric
invariants of the net, generalizing \cite{CHL12}?}

To this end, we start by presenting a general construction of local-entire
cyclic cocycles out of local-exponentially bounded super-KMS functionals for
graded translation-covariant nets over $\R$, which
is mainly due to \cite{BG}; in the main part we then show that these cocycles
are nontrivial and form geometric invariants for ``regular perturbations" of our
dynamical system. As a problem this has been pointed out in \cite[Sec.7]{BG} for
the specific model treated in Example \ref{ex:JLO-FF} and more vaguely already
in \cite[Sec.6\&7]{Lo01}. A subsequent deeper investigation of the
involved geometric invariants, probably related to index and K-theory and
possibly recovering parts of the representation theory as recently achieved in
\cite{CHL12} would be a natural task for future. We should stress that the
setting in \cite{CHL12} is different and analytically substantially easier as we
have a
$\theta$-summable spectral triple and a (bounded) super-Gibbs functional (with
respect to rotations) available.

After a short section on preliminaries as found in more detail in 
\cite[Sec.2]{Hil1}, we provide the cocycle construction in Section
\ref{sec:JLO}, and discuss perturbations, homotopy-invariance, and a brief
example in the final
and main Section \ref{sec:perturb}.

In the present paper, we deal with quantum field nets over $\R$, whose physical
meaning is
that of a chiral component over a light-ray in two-dimensional Minkowski
spacetime. One might, however, in some cases regard them also as
restrictions to one light-ray  of certain special nets over higher-dimensional
manifolds. Moreover, such nets are deeply related to the concept of filtration
in quantum probability, and a connection point and applications to that area are
quite likely, yet beyond the scope of the present article. Thus we expect the
ideas found here to be of interest in a much wider context.

\section{Notation and preliminaries on super-KMS functionals}\label{sec:gen}

Let $\I$ stand for the set of nontrivial bounded open intervals in $\R$. 
For every interval $I\in\I$ we write $|I|:= \sup \{ |x| : x\in I\}$.

A \emph{graded translation-covariant net $\A$ over $\R$} is a map $I \mapsto
\A(I)$ from the set $\I$ to the set of von Neumann algebras on a common
infinite-dimensional separable Hilbert space $\H$ satisfying the 
following properties:
\begin{itemize}
\item[-] \emph{Isotony.} $\A(I_1)\subset \A(I_2)$ if $I_1,I_2\in\I$ and $I_1\subset I_2$.
\item[-] \emph{Grading.} There is a fixed selfadjoint unitary $\Gamma\not=\unit$
(the grading unitary) on $\H$ satisfying
$\Gamma \A(I) \Gamma = \A(I)$ for all $I\in \I$. We write $\gamma=\Ad\Gamma$ and
define the usual \emph{graded commutator}
\begin{align*}
[x,y]=&\; xy+\frac14 (y-\gamma(y))(x-\gamma(x))- \frac14 (y+\gamma(y))(x-\gamma(x))\\
&- \frac14 (y-\gamma(y))(x+\gamma(x))- \frac14 (y+\gamma(y))(x+\gamma(x)), \quad x,y\in\A(I).
\end{align*}
\item[-] \emph{Translation-covariance.} There is a strongly continuous unitary 
representation on $\H$ of the translation group $\R$ with infinitesimal
generator $P$, commuting with $\Gamma$, and such that 
\[
 \rme^{\rmi tP} \A(I) \rme^{-\rmi t P} = \A(t+I), \quad t\in\R, I\in\I ,
\]
and the corresponding point-strongly continuous one-parameter automorphism 
group $(\alpha_t)_{t\in\R}$ (\ie $t\mapsto \alpha_t(x)$ is $\sigma$-weakly
continuous, for every $x\in\A(I)$) restricts to *-isomorphisms from $\A(I)$ to
$\A(t+I)$, for every $t\in\R$ and $I\in\I$, and is asymptotically
graded-abelian:
\[
\lim_{t\ra \infty} [x,\alpha_t(y)] = 0, \quad x,y\in \A(I), I\in\I.
\]
\item[-] \emph{Positivity of the energy.} $P$ is positive.
\end{itemize}

The \emph{universal} or \emph{quasi-local  C*-algebra} corresponding to a 
(graded) net $\A$ over $\R$ is defined as the C*-direct limit
\[
\AA := \lim_{\ra} \A(I)
\]
over $I\in\I$, \cf also \cite{BR2,CLTW1,Haag}, noting that $\I$ is directed,
and its norm will be simply denoted by $\|\cdot\|$. For all $I\in\I$, $\A(I)$
is naturally identified with a subalgebra of $\AA$. Throughout this paper we use
Gothic letters for the quasi-local C*-algebra of the net with corresponding
calligraphic letter. We write $\alpha$ again for the induced group of
automorphisms of $\AA$. 

Let $\A$ be a graded translation-covariant net. A \emph{superderivation} on 
$\AA$ with respect to the grading $\gamma$ and translation group $\alpha$ is a
linear map $\delta:\dom(\delta)\subset \AA\ra \AA$ such that:
\begin{itemize}
\item[$(i)$] $\dom(\delta)\subset \AA$ is an \emph{$\alpha$-$\gamma$-invariant} 
(\ie globally invariant under the action of every $\alpha_t$, $t\in\R$, as well
as $\gamma$) unital *-subalgebra, with 
\[
\alpha_t\circ\delta(x)=\delta\circ\alpha_t(x), \quad \gamma\circ\delta(x)=
-\delta\circ\gamma(x), \quad \delta(x^*)=\gamma(\delta(x)^*), \quad
x\in\dom(\delta), t\in\R;
\]
\item[$(ii)$] $\delta(xy)= \delta(x)y+\gamma(x)\delta(y)$, for all $x,y\in\dom(\delta)$,
\item[$(iii)$] $\delta_I := \delta \restriction_{\dom(\delta)\cap\A(I)}$ is  a
($\sigma$-weakly)-($\sigma$-weakly) closed $\sigma$-weakly densely defined map
with image in $\A(I)$,
\item[$(iv)$] $\Cci(\delta_I) := \bigcap_{n\in\N} \dom(\delta^n_I)  \subset
\dom(\delta_{0})\cap\A(I) \subset \A(I)$ is $\sigma$-weakly dense,
\end{itemize}
By $\dom(\cdot)_I$ we always mean $\dom(\cdot)\cap \A(I)$ and thus 
$\dom(\delta_I)=\dom(\delta)_I$; $\dom(\cdot)_c$ stands for the union over
$I\in\I$ of $\dom(\cdot)_I$, which in some cases may actually be equal to
$\dom(\cdot)$; in particular, $\Cci(\delta)_c = \bigcup_{I\in\I}
\bigcap_{n\in\N} \dom(\delta^n)_I$. We then call
$(\AA,\gamma,(\alpha_t)_{t\in\R},\delta)$ a \emph{graded quantum dynamical
system}. We shall be interested in modifications of the usual KMS condition on
$(\AA,\alpha)$, and we consider only the case of inverse temperature $\beta=1$;
this can always be achieved by rescaling if $\beta\not=0,\infty$.

All *-algebras in this paper are understood to be unital with unit $\unit$ and
all Hilbert spaces separable.

Given $t\in\R$, write
\[
\Delta^t_n:=\{ s\in \R^{n}: 0\le \sgn(t) s_1 \le ...\le \sgn(t) s_n\le |t| \},
\]
$\Delta_n:=\Delta^1_n$ and the tube 
\[
\T^n:= \{s\in\C^n: \Im(s) \in \Delta_n\}.
\]
Notice that $\T^1$ is the standard closed strip in the complex plane.

Super-KMS functionals are some of the central objects of this paper and we 
choose the following definition, which was motivated by the corresponding ones
in \cite{BG,BL00} but is actually stronger and more suitable for the theory and
examples developed in this paper and in \cite{Hil1}.

\begin{definition}\label{def:gen-sKMSfunctional}
A \emph{super-KMS functional} (in short \emph{sKMS functional}) $\phi$ on a
graded quantum dynamical system $(\AA, \gamma, (\alpha_t)_{t\in\R}, \delta)$
for a given translation-covariant net $\A$ is
a linear functional defined on a  *-subalgebra $\dom(\phi)\subset \AA$ such
that:
\begin{itemize}
\item[$(S_0)$] Domain properties: $\phi(x^*)=\overline{\phi(x)}$, for all 
$x\in\dom(\phi)$; $\dom(\phi)_I\subset \A(I)$ is $\sigma$-weakly dense, for all
$I\in\I$, and $\dom(\phi)$ is globally $\alpha$-$\gamma$-invariant.
\item[$(S_1)$] Local normality: $\phi_I:=\phi\restriction_{\dom(\phi)_I}$ is
bounded and extends to a normal (\ie $\sigma$-weakly continuous) linear
functional on $\A(I)$, denoted again $\phi_I$, for all $I\in\I$.
\item[$(S_2)$] sKMS property: for every $x,y\in \dom(\phi)$, there is  a
continuous function $F_{x,y}$ on the strip $\T^1$ which is analytic on the
interior, satisfying
\[
F_{x,y}(t) = \phi(x \alpha_t(y)),\quad F_{x,y}(t+\rmi) =  \phi
(\alpha_t(y)\gamma(x)), \quad t\in\R,
\]
and there are constants $C_0>0$ and $p_0\in2\N$ depending only on $x,y,\phi$ 
such that
\[
|F_{x,y}(t)| \le C_0 (1+|\Re(t)|)^{p_0}, \quad t\in\T^1.
\]
\item[$(S_3)$] Normalization: $\phi(\unit)=1$.
\item[$(S_4)$] Derivation invariance: $\im (\delta_I) \subset \dom(\phi_I)$, 
for all $I\in\I$, and $\phi\circ \delta =0$ on $\Cci(\delta)_c$.
\item[$(S_5)$] Weak supersymmetry: for every $x,z\in\dom(\phi)_c$ and 
$y\in\Cci(\delta)_c$, we have
\[
\phi(x\delta^2(y)z) = -\rmi\frac{\rmd}{ \rmd t} \phi(x\alpha_t(y)z)\restriction_{t=0}.
\]
\end{itemize}
Some sKMS functionals exhibit the following additional property:
\begin{itemize}
\item[$(S_6)$] Local-exponential boundedness: there are suitable
constants $C_1,C_2>0$ such that $\|\phi \restriction_{\dom(\phi)_I}\|$ is
bounded by $C_1\rme^{C_2 |I|^2}$, for every $I\in\I$. \end{itemize}
\end{definition}

Let us collect the following general properties from \cite[Sec.2]{Hil1}.

\begin{theorem}\label{prop:gen-obstruction}
Let $(\A,\gamma,\alpha)$ be a graded translation-covariant net and
\linebreak $(\phi,\dom(\phi))$ a functional on $(\AA,\gamma,\alpha)$ satisfying
$(S_0)$-$(S_3)$. Then the following holds:
\begin{itemize}
\item[$(1)$] $\phi$ is translation and grading invariant, \ie
\[
\phi\circ\alpha_t=\phi = \phi\circ \gamma, \quad t\in\R.
\]
\item[$(2)$] $\phi$ is neither positive nor bounded.
\item[$(3)$] The  functionals $|\phi_I|$ and $\phi_I^\pm:= \frac12  (|\phi_I|\pm
\phi_I)$ obtained through restriction are individually well-defined, bounded and
positive, but they do not form a directed system with respect to restriction, so
they do not give rise to positive (unbounded) functionals on $\AA$.
\end{itemize}
\end{theorem}

The original motivation for sKMS functionals comes from the studies of
supersymmetric dynamics and ``phase transitions" between bosons and fermions.
This is discussed within the context of the free field model in
\cite[Sec.3]{Hil1}, together with a conditional uniqueness and existence proof.
Our main application of sKMS functionals here lies in local-entire cyclic
cohomology as carried out in the next section.

\section{JLO cocycles for super-KMS functionals}\label{sec:JLO}

Given an sKMS functional for a graded quantum dynamical system, Jaffe, 
Lesniewski and Wisniowski have found a natural way to associate an entire cyclic
cocycle \cite{JLW}, generalizing their famous construction of JLO cocycles for
super-Gibbs functionals with supercharges \cite{JLO1}. Unfortunately, \cite{JLW}
works only for bounded sKMS functionals (\eg as is the case with the rotation
group in \cite[Sec.4]{Hil1} for nets over $S^1$), which for the translation
group under usual assumptions do not exist according to Theorem
\ref{prop:gen-obstruction}(2). Buchholz and Grundling showed, however, that in a
special model a local-exponential bound similar to $(S_6)$ is satisfied, which
permits to associate a (local-)entire cyclic cocycle again, although involving
certain analytical difficulties \cite[Sec.6]{BG}. Their result and proof
actually generalize to sKMS functionals as in Definition
\ref{def:gen-sKMSfunctional} for a generic graded quantum dynamical systems.
Henceforth, let $\A$ be a generic graded translation-covariant net $\A$ on $\R$
and $((\AA,\|\cdot\|),\gamma,(\alpha_t)_{t\in\R},\delta)$ a corresponding
generic graded quantum dynamical system.

Given a normed algebra $(A,\|\cdot\|')$, we recall that the induced norm of  an
$n$-linear functional $\rho_n$ on $A$ is
\[
 \|\rho_n\|' = \sup_{x_i\in A} \frac{|\rho_n(x_0,...,x_n)|}{\|x_0\|'\cdot ...
\cdot \|x_n\|'}. 
\]

\begin{definition}\label{def:JLO-ECC}
A \emph{local-entire cochain} on a *-subalgebra $A\subset\AA$ is given by a 
sequence $(\rho_n)_{n\in\NN}$ of $(n+1)$-linear maps $\rho_n$ on $A$ with
$\rho_n(x_0,\ldots,x_n)=0$ if $x_i\in\C\unit$ for some $i=1,\ldots,n$, such
that
\[
\lim_{n\ra\infty} n^{1/2} \|\rho_n\restriction_{A\cap\A(I)}\|^{1/n} =0, \quad
I\in\I. 
\]
The even cochains are those with $\rho_{2n+1}=0$ and the odd cochains those with
$\rho_{2n}=0$, for all $n\in\NN$. A local-entire cochain is a \emph{local-entire
cyclic cocycle} if $\partial \rho =0$, where $\partial:= B+b$ maps even into odd
local-entire chains and v.v., and 
\begin{align*}
(b \rho)_n(x_0,..,x_n)
:=& \sum_{j=0}^{n-1} (-1)^{j} \rho_{n-1}(x_0,...,x_jx_{j+1},...,x_{n}) 
+ (-1)^n \rho_{n-1}(x_nx_0,x_1,...,x_{n-1})\\
(B \rho)_n(x_0,...,x_n) 
:=& \sum_{j=0}^n (-1)^{nj} \rho_{n+1}(\unit,x_j,...,x_{j-1}), \quad x_i\in A.
\end{align*}
\end{definition}

We consider here for $A$ in the above definition the *-algebra
$\Cci(\delta)^\gamma_c$ (the fixed points of $\Cci(\delta)_c$ under $\gamma$)
equipped with the graph norm $\|\cdot\|_* := \|\cdot \| + \|\delta(\cdot)\|$.

\begin{theorem}\label{prop:JLO-cocycle}
Given a local-exponentially bounded sKMS functional $\phi$ for the graded 
quantum dynamical system $(\AA,\gamma,(\alpha_t)_{t\in\R},\delta)$, the
expression
\[
\tau_n(x_0,...,x_n) := \acont_{t\ra\rmi} \int_{\Delta_n} \phi(x_0  \alpha_{
s_1t}(\delta(x_1))...\alpha_{s_n t}(\delta(x_n))) \rmd^n s, \quad
x_i\in\Cci(\delta)_c^\gamma,
\]
for even $n\in\NN$, and $\tau_n=0$, for odd $n\in\NN$, is welldefined,  and
gives rise to an even local-entire cyclic cocycle $(\tau_n)_{n\in\NN}$ on
$\Cci(\delta)_c^\gamma$, called the \emph{JLO cocycle}.
\end{theorem}

\begin{proof}
The proof is basically given in \cite[Th.6.3\&6.4]{BG}, but for the reader's 
convenience and since some aspects of our setting are slightly different and
required again in the following section, we include it here with the
corresponding adjustments. The first part deals with local-entireness, the
second with the algebraic cocycle condition.

(Part 1.) Given $n\in\NN$ and $x_i\in\dom(\phi)_c$, $i=0,\ldots,n$, as in the
assumptions, there is $I\in\I$ such that all $x_i$ lie in $\dom(\phi)_I$.
We fix such an $I$ and write $A:=\dom(\phi)_I$ in the first part of the proof.
Then, for every $t_i\in\R$, we have $\alpha_{t_i}(x_i)\in\A(I+t_i)$ and
$\|\alpha_{t_i}(x_i)\|= \|x_i\|$ by the local *-property of $\alpha_{t_i}$;
property $(S_6)$ implies then
\[
\Big|\phi\big(\alpha_{t_0}(x_0)\cdots\alpha_{t_n}(x_n)\big)\Big|
\leq C_1 \rme^{C_2 (|I|+\sup_i |t_i|)^2}\|x_0\|\cdots\|x_n\|.
\]
We would like to find an upper bound for this in terms of analytic functions.
One can easily check that $(|I|+ \sup_i |t_i|)^2 \le
|I|^2+2|I|+(1+2|I|)\sum_{i=1}^n t_i^2$, so
\[
\Big|\phi\big(\alpha_{ t_0}(x_0)\cdots\alpha_{ t_n}(x_n)\big)\Big|
\leq B_1\exp\Big(B_2\sum_{i=1}^n t_i^2\Big)\|x_0\|\cdots\|x_n\|
\]
with constants $B_1=C_1 \rme^{C_2 (|I|^2+2|I|)}$ and $B_2=C_2(1+2|I|)$.
We are interested in an analytic continuation of this as a function of $t$.  For
all $s_i\in\R$, define:
\[
F_{x_0,\ldots ,x_n}(s_1,\ldots,s_n):=
\exp\Big(-B_2 \sum_{k=1}^n(s_1+\cdots+s_k)^2\Big)\cdot 
\phi\Big(x_0\alpha_{s_1}(x_1)\cdots\alpha_{s_1+\cdots+s_n}(x_n)\Big).
\]
Consequently, $|F_{x_0,\ldots ,x_n}(s_1,\ldots,s_n)|\leq
B_1\|x_0\|\cdots\|x_n\|$,  for all $s\in\R^n$, and by the sKMS property $(S_2)$
of $\phi$, the function $F_{x_0,\ldots ,x_n}$ can be analytically continued in
each variable $s_j$ to the strip $\T^1$. This way, we obtain functions
$F^{(j)}_{x_0,\ldots ,x_n}$ which are analytic in the flat tubes
$T^{(j)}:=\R^{j-1}\times \T^1\times\R^{n-j}$, for all $j=1,\ldots,n$; using the
flat tube theorem \cite[Lem.A.2]{BB} inductively, we thus obtain a unique
analytic continuation of $F_{x_0,\ldots ,x_n}$ into the tube $T:=\{z\in \C^n:
0\le \Im z_j \le 1,\; j=1,...,n; \; \sum_{j=1}^n \Im z_j \le 1\}$ (the convex
hull of $\bigcup_{j=1,\ldots,n}T^{(j)}$) coinciding with each
$F^{(j)}_{x_0,\ldots ,x_n}$ on $T^{(j)}$. For several purposes like the
entireness condition we need a bound for this analytic function $F_{x_0,\ldots
,x_n}$ on $T$, and we start by finding bounds for $F_{x_0,\ldots ,x_n}^{(j)}$. 
Let $G_{x_0,\ldots,x_n}(s_1,\ldots,s_n):={\phi\Big(x_0\alpha_{s_1}(x_1) 
\cdots\alpha_{s_1+\cdots+s_n}(x_n)\Big)}$ which has just been shown to have a
unique analytic continuation to each $T^{(j)}$, and by the growth condition in
$(S_2)$ we know that $\big|G_{x_0,\ldots,x_n}(s_1,\ldots,s_j+\rmi
r_j,\ldots,s_n)\big|\leq C_0(1+|s_j|)^{p_0}$ with certain scalars $C_0\in\R_+$
and $p_0\in\N$ independent of $s_j$ and $r_j\in[0,1]$. Then by the above
definition
\begin{equation}\label{eq:JLO-Fest}
\begin{aligned}
F_{x_0,\ldots ,x_n} ({s_1,\ldots,s_j+\rmi r_j,\ldots,s_n})
=& G_{x_0,\ldots,x_n}({ s_1,\ldots,s_j+\rmi r_j,\ldots,s_n}) \times \\
&\times \exp\big(B_2 r_j^2(n+1-j)- B_2\sum_{k=1}^n 
(s_1+\cdots+s_k)^2+\rmi\theta_{s,r}\big),
\end{aligned}
\end{equation}
for some $\theta_{s,r}\in\R$.
From the above polynomial growth property of $G_{x_0,\ldots,x_n}$ in $s_j$  we
conclude that $F^{(j)}_{x_0,\ldots ,x_n}$ is bounded. Thus according to the
Phragmen-Lindel\"of theorem, \cf \cite[Prop.5.3.5]{BR2}, the bound of the
analytic function $F^{(j)}_{ x_0,\ldots ,x_n}$ is attained on the boundary of
$T^{(j)}$. On the real part of the boundary 
of $T^{(j)}$ we have computed above:
$\big|F^{(j)}_{ x_0,\ldots ,x_n}(s_1,\ldots,s_n)\big| \leq B_1\|x_0\|\cdots\|x_n\|$.
By the sKMS property $(S_2)$ and translation invariance of $\phi$ we have on the
other part $\R^{j-1}\times(\rmi+\R)\times\R^{n-j}$:
\begin{align*}
\big|G_{x_0,\ldots,x_n}(s_1, \ldots, & s_j+\rmi,\ldots,s_n)\big|\\
=& \Big|\phi\Big(\alpha_{ s_1+\cdots s_j}(x_j)\cdots \alpha_{ s_1+\cdots+ s_n}
(x_n) x_0 \alpha_{ s_1}(x_1)\cdots\alpha_{ s_1+\cdots+s_{j-1}}
(x_{j-1})\Big)\Big| \\
\leq& B_1\exp\Big(B_2\sum_{k=1}^n (s_1+\cdots+s_k)^2\Big)\cdot\|x_0\|\cdots\|x_n\|,
\end{align*}
hence by \eqref{eq:JLO-Fest}:
\[
\big|F^{(j)}_{ x_0,\ldots ,x_n}(s_1,\ldots,z_j,\ldots,s_n)\big|  \leq
B_1\rme^{B_2 n}\|x_0\|\cdots\|x_n\|,\quad z_j\in \T^1.
\]

We have to show that this bound holds on all $T$. To this end let 
$C:=B_1\rme^{B_2 n}\|x_0\|\cdots\|x_n\|$ and define, for every
$\alpha\in[0,2\pi]$: $f_\alpha(z_1,\ldots,z_n):=\big( F_{ x_0,\ldots
,x_n}(z_1,\ldots,z_n) -\rme^{\rmi\alpha}C\big)^{-1}$, for all $z\in T$. Since
$z_j\mapsto f_\alpha (s_1,\ldots,z_j,\ldots,s_n)$ is analytic on the strip
$\T^1$, for every $j=1,\ldots,n$, the flat tube theorem \cite[Lem.A.2]{BB} again
implies that $f_\alpha$ has a unique analytic continuation to $T$, and hence
cannot have any singularities in  $T$, \ie $F_{ x_0,\ldots ,x_n}(z_1,\ldots,z_n)
\not=\rme^{\rmi\alpha}C$ for all $\alpha$. Since $F_{ x_0,\ldots ,x_n}$ is
continuous, its image set $F_{ x_0,\ldots ,x_n}(T)$ must be connected, and since
it has some points inside the circle of radius $C$, the entire image set is
inside that circle, so
\[
\big|F_{ x_0,\ldots ,x_n}(z_1,\ldots,z_n)\big|\leq B_1 \rme^{B_2
n}\|x_0\|\cdots\|x_n\|, \quad z\in T,\; x_i\in A.
\]
We finally perform a change of variables and summarize the above as follows:

\emph{For all $I\in\I$ and $x_i\in  \dom(\phi)_I$, the function
\begin{equation}\label{eq:JLO-analcont}
s\in\R^n \mapsto \phi \big( x_0\alpha_{s_1}(x_1)\cdots \alpha_{s_n}(x_n))
\end{equation}
has a unique analytic continuation to the tube $\T^n$, for which we write shortly
\[
z\in\T^n \mapsto \phi\big(x_0\alpha_{z_1}(x_1)\cdots \alpha_{z_n}(x_n)\big),
\]
although $\alpha_{z_i}(x_i)$ itself makes no sense for complex $z_i$.  The
continuation is bounded by
\begin{equation}\label{eq:JLO-exp-est}
|\phi\big(x_0\alpha_{z_1}(x_1)\cdots \alpha_{z_n}(x_n)\big)| \le C_1  \rme^{2C_2
(|I|+1)^2 (n+1)} \rme^{2C_2 (|I|+1) \sum_{k=1}^n |z_k|^2} \|x_0\|\cdots \|x_n\|,
\end{equation}
for all $z\in \T^n$.}

With the special choice $(z_1,\ldots,z_n)=\rmi(r_1,\ldots,r_n)\in \rmi\Delta_n$,
 we arrive at
\[
\big|\tau_n(x_0,\ldots,x_n)\big| \leq
\frac{C_1}{n!}\rme^{4C_2(|I|+1)^2(n+1)}\|x_0\|\;
\|\delta(x_1)\|\cdots\|\delta(x_n)\|
\]
since the volume of $\Delta_n$ is $1/n!$. Thus 
\[
n^{1/2} \|\tau_n\|_*^{1/n} \le n^{1/2}(C_1/n!)^{1/n} 
\rme^{4C_2(|I|+1)^2(n+1)/n} \sim n^{-1/2} C_1 \rme^{4C_2(|I|+1)^2} \ra 0,\quad
n\ra \infty,
\]
which concludes the proof of local-entireness.

(Part 2.) In order to prove the algebraic cocycle condition, we first claim that
\begin{equation}\label{eq:JLO-BGLem}
\begin{aligned}
 \phi\Big(\delta(x_0) \alpha_{\rmi s_1} & (\delta (x_1))\cdots\alpha_{\rmi
s_{n}}(\delta 
(x_n))\Big)
= \sum_{j=1}^n \frac{\partial}{\partial s_j}\phi\Big((\gamma (x_0))\alpha_{\rmi
s_1}(\gamma\delta (x_1))\cdots\\ 
&\cdots \alpha_{\rmi s_{j-1}}(\gamma\delta (x_{j-1}))\alpha_{\rmi
s_j}(x_j)\alpha_{\rmi s_{j+1}}(\delta (x_{j+1}))\cdots \alpha_{\rmi
s_{n}}(\delta (x_n))\Big), 
\end{aligned}
\end{equation}
for $x_i\in\Cci(\delta)_I$ and $s\in\Delta_n$.

\emph{Proof of the claim.} Recall that the superderivation property $(ii)$ implies
\[
\delta(x_0)\cdots \delta(x_n)= \delta(x_0\delta(x_1) \cdots  \delta(x_n)) 
- \sum_{j=1}^n \gamma(x_0 \delta(x_1)\cdots \delta(x_{j-1})) \delta^2(x_j) 
\delta(x_{j+1}) \cdots  \delta(x_n),
\]
and that $\alpha_{t_j}$ commutes with $\delta$.
Combining this first with property $(S_4)$ and then with $(S_5)$, we find
\begin{align*}
\phi\Big(& \delta(x_0)  \alpha_{t_1} (\delta 
(x_1))\cdots\alpha_{t_{n}}(\delta(x_n))\Big) \\
=& -\sum_{j=1}^n \phi\Big(\gamma (x_0)\alpha_{t_1} (\gamma\delta
(x_1))\cdots \alpha_{t_{j-1}}(\gamma\delta (x_{j-1}))\alpha_{t_j}(\delta^2
(x_j))\alpha_{t_{j+1}}(\delta (x_{j+1}))\cdots \alpha_{t_{n}}(\delta
(x_n))\Big)\\
=& \sum_{j=1}^n \rmi\frac{\partial}{\partial t_j}  \phi\Big((\gamma
(x_0))\alpha_{t_1}(\gamma\delta (x_1))\cdots \alpha_{t_{j-1}}(\gamma\delta
(x_{j-1}))\alpha_{t_j}(x_j)\alpha_{ t_{j+1}}(\delta (x_{j+1}))\cdots
\alpha_{t_{n}}(\delta (x_n))\Big).
\end{align*}
As in \eqref{eq:JLO-analcont}, the two functions
{\small
\begin{align*}
t\in\R^d \mapsto& \phi\Big(\delta(x_0)  \alpha_{t_1} (\delta (x_1))
\cdots\alpha_{t_{n}}(\delta(x_n))\Big),\\
t\in\R^d \mapsto& \sum_{j=1}^n \rmi\frac{\partial}{\partial t_j} 
\phi\Big((\gamma (x_0))\alpha_{t_1}(\gamma\delta (x_1))\cdots
\alpha_{t_{j-1}}(\gamma\delta (x_{j-1}))\alpha_{t_j}(x_j)\alpha_{
t_{j+1}}(\delta (x_{j+1}))\cdots \alpha_{t_{n}}(\delta (x_n))\Big)
\end{align*}
}
extend uniquely to analytic functions on the tubes $\T^n$ coinciding on $\R^n$. 
Thus they extend to the same function on $\T^n$, with the same argument as in
\eqref{eq:JLO-analcont}. With the special choice $z=\rmi s\in \rmi \Delta_n$,
we obtain \eqref{eq:JLO-BGLem}, thus the claim is proved.

The sKMS condition $(S_2)$ and analyticity property  \eqref{eq:JLO-analcont} 
together yield the following relations (for a detailed proof \cf
\cite[Lem.8.5]{BG} replacing the algebra denoted there  by $\A_0$ with
$A=\Cci(\delta)\cap\A(I)$ in the present setting):

\begin{lemma}
\label{lem:JLO-cocycle}
In the above setting and with $x_i\in A$, the following equalities hold:
\begin{itemize}
\item[$(i)$]
$\int_{\Delta_n}\phi\Big(x_0\alpha_{{\rmi s_1}}(x_1)\cdots\alpha_{{\rmi s_n}}
(x_n)\Big)\rmd^n s=\int_{\Delta_n}\phi\Big(\gamma(x_n)\alpha_{{\rmi
s_1}}(x_0)\cdots\alpha_{{\rmi s_n}}(x_{n-1})\Big)\rmd^n s.$
\item[$(ii)$] For $j=2,\ldots,n$ we have:
\begin{equation}\label{eq:JLO-cocycle1}
\begin{aligned}
\int_{\Delta_{n+1}} & \frac{\partial}{\partial s_j}\phi\Big(x_0\alpha_{{\rmi
s_1}} (x_1)\cdots\alpha_{{\rmi s_{n+1}}}(x_{n+1})\Big)\rmd^{n+1} s \\
=&\int_{\Delta_n}\Big(\phi\Big(x_0\alpha_{{\rmi s_1}}(x_1)\cdots\alpha_{ \rmi
s_j} (x_jx_{j+1})\cdots\alpha_{{\rmi s_n}}(x_{n+1})\Big)\\
&-\phi\Big(x_0\alpha_{{\rmi s_1}}(x_1)\cdots\alpha_{ \rmi
s_{j-1}}(x_{j-1}x_j)\cdots\alpha_{{\rmi s_n}}(x_{n+1})\Big)\Big)\rmd^n s, \\ 
\end{aligned}
\end{equation}
\begin{equation}\label{eq:JLO-cocycle2}
\begin{aligned}
\int_{\Delta_{n+1}}& \frac{\partial}{\partial s_1}\phi\Big(x_0\alpha_{{\rmi
s_1}}(x_1)\cdots\alpha_{{\rmi s_{n+1}}}(x_{n+1})\Big)\rmd^{n+1} s \\ 
=&\int_{\Delta_n}\Big(\phi\Big(x_0\alpha_{{\rmi s_1}}(x_1x_2)\alpha_{ \rmi
s_{2}}(x_3)\cdots\alpha_{{\rmi s_n}}(x_{n+1})\Big)\\ 
&-\phi\Big(x_0x_1\alpha_{{\rmi s_1}}(x_2)\cdots\alpha_{{\rmi
s_n}}(x_{n+1})\Big)\Big)\rmd^n s.\\ 
\end{aligned}
\end{equation}
\end{itemize}
\end{lemma}

Let us continue now with the proof of the cocycle condition, starting  with
$B\tau_{n+1}$. Recalling that $\tau$ is an even chain, we may restrict to odd
$n$. From the definition of $B$ we then have, for $x_i\in A^\gamma$ and using
$\delta(\unit)=0$:
\begin{align*}
\big(B  {\tau}_{n+1}\big)&(x_0,\ldots,x_n)\\
=&\int_{\Delta_{n+1}}\sum_{j=0}^n (-1)^{nj}\phi\Big(\unit  \alpha_{ \rmi
s_1}(\delta(x_j))\ldots,\alpha_{ \rmi s_{n-j+1}}(\delta(x_n))\alpha_{ \rmi
s_{n-j+2}}(\delta(x_0))\ldots\\
&\ldots \alpha_{ \rmi s_{n+1}}(\delta(x_{j-1}))\Big) \rmd^{n+1} s.
\end{align*}
Applying Lemma \ref{lem:JLO-cocycle}(i) repetitively and using  the fact that
$\gamma(\delta(x_j))=-\delta(x_j)$ for all $j$ and that $(-1)^{nj}=(-1)^{j}$ for
odd $n$, we obtain
\begin{equation}\label{eq:JLO-cocycleB}
\begin{aligned}
\big(B  {\tau}_{n+1}\big)&(x_0,\ldots,x_n)\\
=&\int_{\Delta_{n+1}}\sum_{j=0}^n
\phi\Big(\delta(x_0) \alpha_{ \rmi s_1}(\delta(x_1))\ldots \alpha_{ \rmi
s_{j-1}}(\delta(x_{j-1}))\alpha_{ \rmi s_{j}}(\unit) \alpha_{ \rmi
s_{j+1}}(\delta(x_{j}))\ldots \\
&\ldots \alpha_{ \rmi s_{n+1}}(\delta(x_n))\Big) \rmd^{n+1} s\\
=& \int_{\Delta_{n}}
\phi\Big(\delta(x_0) \alpha_{ \rmi s_1}(\delta(x_1))\ldots \alpha_{ \rmi
s_{j-1}}(\delta(x_{j-1}))\alpha_{ \rmi s_{j}}(\delta(x_{j}))\ldots\alpha_{ \rmi
s_{n}}(\delta(x_n))\Big) \rmd^{n} s
\end{aligned}
\end{equation}

Now we consider $b\tau_{n-1}$, again with odd $n$ of course.  By definition of
$b$, we have
\begin{equation}\label{eq:JLO-cocycleb}
\begin{aligned}
\big(b\tau_{n-1}\big) & (x_0,\ldots,x_n)\\
=&\int_{\Delta_{n-1}} \Big(\sum_{j=0}^{n-1}(-1)^j\phi\Big(x_0\alpha_{ \rmi
s_1}(\delta (x_1)) 
\cdots \alpha_{ \rmi s_j}(\delta(x_j x_{j+1}))\cdots\alpha_{ \rmi
s_{n-1}}(\delta (x_n))\Big)\\ 
&-\phi\left(x_nx_0
\alpha_{ \rmi s_1}(\delta (x_1))\cdots \alpha_{ \rmi s_{n-1}}(\delta
(x_{n-1}))\right)\Big) \rmd^{n-1} s. 
\end{aligned}
\end{equation}
For the contributions with $j=0,1$ we find
\begin{align*}
\int_{\Delta_{n-1}} & \phi\Big(x_0x_1\alpha_{ \rmi s_1}(\delta (x_2))\cdots
\alpha_{ \rmi s_{n-1}}(\delta (x_n))\Big)\rmd^{n-1} s   \\ 
&-\int_{\Delta_{n-1}}\phi\Big(x_0\alpha_{ \rmi s_1}(\delta(x_1x_2))\alpha_{ \rmi
s_2}(\delta (x_3))\cdots \alpha_{ \rmi s_{n-1}}(\delta (x_n))\Big)  \rmd^{n-1} s
\\ 
=&-\int_{\Delta_{n-1}} \Big(\phi\Big(x_0 \alpha_{ \rmi s_1}\big
(\delta(x_1)x_2+x_1\delta (x_2)\big)\alpha_{ \rmi s_2}(\delta (x_3))\cdots
\alpha_{ \rmi s_{n-1}}(\delta(x_n))\Big)   \\
&- \phi\Big(x_0 x_1\alpha_{ \rmi s_1}(\delta (x_2))\alpha_{ \rmi s_2}(\delta
(x_3))\cdots \alpha_{ \rmi s_{n-1}}(\delta(x_n))\Big)\Big) \rmd^{n-1} s  \\  
=&-\int_{\Delta_{n-1}}\phi\Big(x_0\alpha_{ \rmi s_1}\big(\delta(x_1)
x_2\big)\alpha_{ \rmi s_2}(\delta (x_3))\cdots \alpha_{ \rmi
s_{n-1}}(\delta(x_n))\Big) \rmd^{n-1} s  \\ 
&-\int_{\Delta_{n}} \frac{\partial}{\partial s_1}\phi\Big(x_0\alpha_{ \rmi
s_1}(x_1)\alpha_{ \rmi s_2}(\delta (x_2))\alpha_{ \rmi s_3}(\delta
(x_3))\cdots\alpha_{ \rmi s_{n}}(\delta (x_n))\Big)\rmd^n s 
\end{align*}
where we made use of \eqref{eq:JLO-cocycle2} in the last step.
For $2\le j\le n-1$ we find, using \eqref{eq:JLO-cocycle1}:
\begin{align*}
(-1)^j & \int_{\Delta_{n-1}}\phi\Big(x_0\alpha_{ \rmi s_1}(\delta 
(x_1))\cdots\alpha_{ \rmi s_j}\big(\delta (x_jx_{j+1})\big)\cdots\alpha_{ \rmi
s_{n-1}}(\delta (x_n))\Big) \rmd^{n-1} s \\
=&(-1)^j\int_{\Delta_{n-1}}\phi\Big(x_0\alpha_{ \rmi s_1}(\delta 
(x_1))\cdots\alpha_{ \rmi s_j}\big(\delta (x_j) x_{j+1}+x_j\delta
(x_{j+1})\big)\cdots\\
&\quad \cdots \alpha_{ \rmi s_{n-1}}(\delta (x_n))\Big) \rmd^{n-1} s\\
=&(-1)^j\int_{\Delta_{n-1}}\Big(\phi\Big(x_0\alpha_{ \rmi s_1}(\delta 
(x_1))\cdots\alpha_{ \rmi s_j}\big(\delta (x_j) x_{j+1}\big)\cdots\alpha_{ \rmi
s_{n-1}}(\delta (x_n))\Big) \\
&+ \phi\Big(x_0\alpha_{ \rmi s_1}(\delta (x_1))\cdots\alpha_{ \rmi s_{j-1}}
\big(\delta (x_{j-1})x_j\big)\cdots\alpha_{ \rmi s_{n-1}}(\delta
(x_n))\Big)\Big) \rmd^{n-1} s\\
&+(-1)^j\int_{\Delta_{n}} \frac{\partial}{\partial s_j}\phi \Big(x_0\alpha_{
\rmi s_1}(\delta(x_1))\cdots\alpha_{ \rmi s_j}(x_j)\cdots\alpha_{ \rmi
s_{n}}(\delta (x_n))\Big)\rmd^n s.
\end{align*}
In these equations, the first and second term on the right-hand side cancel
between subsequent summands of the sum \eqref{eq:JLO-cocycleb} over $j$.
Putting all together, we obtain 
\begin{align*}
\big(b  \tau_{n-1}\big)&(x_0,\ldots,x_n)\\
=& (-1)^{n-1}\int_{\Delta_{n-1}}\phi\Big(x_0\alpha_{ \rmi s_1}(\delta (x_1))
\cdots\alpha_{ \rmi s_{n-1}}\big(\delta (x_{n-1})x_n\big)\Big)\rmd^{n-1} s \\
&+\sum_{j=1}^{n-1}(-1)^j\int_{\Delta_{n}} \frac{\partial}{\partial s_j}\phi
\Big(x_0\alpha_{ \rmi s_1}(\delta (x_1))\cdots\alpha_{ \rmi
s_j}(x_j)\cdots\alpha_{ \rmi s_{n}}(\delta (x_n))\Big)\rmd^n s \\
&+(-1)^n\int_{\Delta_{n-1}}\phi\left(x_n x_0\alpha_{ \rmi s_1}(\delta (x_1))
\cdots \alpha_{ \rmi s_{n-1}}(\delta (x_{n-1}))\right) \rmd^{n-1} s \\
=&\sum_{j=1}^n(-1)^j\int_{\Delta_{n}}  \frac{\partial}{\partial s_j}\phi\Big
(x_0\alpha_{ \rmi s_1}(\delta (x_1))\cdots\alpha_{ \rmi s_j}( x_j)\cdots\alpha_{
\rmi s_{n}}(\delta (x_n))\Big)\rmd^n s\\
=&-\int_{\Delta_{n}}\sum_{j=1}^n \frac{\partial}{\partial s_j} \phi\Big(x_0
\alpha_{ \rmi s_1}(\gamma\delta (x_1))\cdots\alpha_{ \rmi s_{j-1}}(\gamma\delta
(x_{j-1}))\alpha_{ \rmi s_j}( x_j)\cdots\alpha_{ \rmi s_{n}}(\delta
(x_n))\Big)\rmd^n s.
\end{align*}
Finally combining this with \eqref{eq:JLO-BGLem} and \eqref{eq:JLO-cocycleB} proves
\begin{align*}
\big(b\tau_{n-1}\big)  (x_0,\ldots,x_n)
=& -\int_{\Delta_{n}}\phi\Big(\delta(x_0)\alpha_{ \rmi s_1}(\delta (x_1))\cdots
\alpha_{ \rmi s_{n}}(\delta (x_n))\Big)\rmd^n s\\
=& -\big(B\tau_{n+1}\big)(x_0,\ldots,x_n), \qquad x_i\in A^\gamma,
\end{align*}
\ie $(B+b)\tau=0$. Since this holds for every choice of $I\in\I$, we see that 
$\tau$ is a local-entire cyclic cocycle on $\Cci(\delta)^\gamma_c$.

\end{proof}


\section{Perturbations of super-KMS functionals and homotopy-invariance of 
their JLO cocycles}\label{sec:perturb}

Let us study the general situation of a perturbation of a given graded  quantum
dynamical system $(\AA,\gamma,(\alpha_t)_{t\in\R},\delta)$ by an odd
selfadjoint $Q\in\Cci(\delta)_c$. Some of the ideas followed here are
found in \cite{JLW}, which however works in an analytically different context.
Let us start by making our concepts of perturbation more precise:

\begin{proposition}\label{prop:JLO-alpha}
Let $\phi$ be a local-exponentially bounded sKMS functional for a graded 
quantum dynamical system $(\AA,\gamma,(\alpha_t)_{t\in\R},\delta)$. For every
$I\in\I$ and odd selfadjoint $Q\in\Cci(\delta)_I$ and $r\in[0,1]$, let
$\delta_r:=\delta + r [Q,\cdot]$ and $a_r:= r\delta(Q) +r^2
Q^2\in\Cci(\delta)_I$, so $\delta_r^2=\delta^2+\ad a_r$. Define formally
\[
 \alpha^r_t (x) := \sum_{n\in\NN} (\rmi t)^n \int_{\Delta_n} 
\ad(\alpha_{s_1t}(a_r))\cdots \ad(\alpha_{s_nt}(a_r)) (\alpha_t(x)) \rmd^n s, 
\]
and
\[
 \gamma^r_t(x) := \sum_{n\in\NN} (\rmi t)^n \int_{\Delta_n} 
\alpha_{s_1t}(a_r)\cdots \alpha_{s_nt}(a_r) \alpha_t(x) \rmd^n s, \quad x
\in\dom(\phi)_c, t\in\R.
\]
Then the sums converge and define one-parameter groups, which commute with 
$\gamma$ and are continuous in the following sense:
\[
-\rmi\frac{\rmd}{\rmd t} \phi(x\alpha^r_t(y)z)\restriction_{t=0} = \phi(x 
(\delta^2 + \ad a_r)(y)z),\quad
-\rmi\frac{\rmd}{\rmd t} \phi(x\gamma^r_t(y)z)\restriction_{t=0} = \phi(x 
(\delta^2+a_r)(y)z),
\]
for every $x,z\in\dom(\phi)_c$ and $y\in\Cci(\delta)_c$.
Moreover, $\alpha^r_t$ are *-automorphisms, and we have the following behavior 
under perturbation: there are constants $C_1,C_2>0$ such that
\begin{equation*}
\|\alpha_t^r(x) -\alpha_t^q(x)\|, \|\gamma_t^r(x)-\gamma_t^q(x)\|
\le 2|q-r| (\|\delta(Q)\| +\| Q^2\|) |t| \rme^{2(\|\delta(Q)\|+\|Q^2\|)} \|x\|,
\end{equation*}
for every $t\in\R$, $x\in\dom(\phi)_c$, and $q,r\in[0,1]$.
\end{proposition}

\begin{proof}
We consider only $\alpha^r_t$ since $\gamma^r_t$ can be treated analogously. 
Since $\|\alpha_s(a_r)\|=\|a_r\|$ is uniformly bounded in $r\in[0,1]$ and 
$s\in\R$, we see that each summand is bounded by $\frac{1}{n!}|t|^n \|2 a_r\|^n
\|x\|$, so the sum converges. In case of $\alpha^r_t$, it follows moreover from
the *-property of $\alpha_t$ and selfadjointness of $a_r$ that $\alpha^r_t$ has
the *-property.

We want to check the group property of $\alpha^r$:
\begin{align*}
\alpha^r_{t_1} \alpha^r_{t_2}(x) 
=& \sum_{n_1\in\NN}\sum_{n_2\in\NN} (\rmi t_1)^{n_1}(\rmi t_2)^{n_2} 
\int_{\Delta_{n_1}}\int_{\Delta_{n_2}} 
\ad(\alpha_{t_1s_1}(a_r))\cdots \ad(\alpha_{t_1s_{n_1}}(a_r))\\ 
& \ad(\alpha_{t_1+t_2s_{n_1+1}}(a_r))\cdots \ad(\alpha_{t_1 + t_2 
s_{n_1+n_2}}(a_r)) \alpha_{t_1+t_2}(x) \rmd^{n_2} s \rmd^{n_1}s \\
=& \sum_{n\in\N} \sum_{n_1=0}^n (\rmi t_2)^n (t_1/t_2)^{n_1} 
\int_{\Delta_{n_1}} \int_{\Delta_{n-n_1}} 
\ad(\alpha_{t_1s_1}(a_r))\cdots \ad(\alpha_{t_1 s_{n_1}}(a_r))\\
& \ad(\alpha_{t_1+t_2 s_{n_1+1}}(a_r))\cdots \ad(\alpha_{t_1 +  t_2 s_{n}}(a_r))
\alpha_{t_1+t_2}(x) \rmd^{n-n_1} s \rmd^{n_1}s \\
=& \sum_{n\in\NN}  (\rmi (t_1+t_2))^n \int_{\Delta_{n}}
\ad(\alpha_{(t_1+t_2)s_1}(a_r))\cdots \ad(\alpha_{(t_1+t_2)s_n}(a_r))
\alpha_{t_1+t_2}(x)\rmd^{n} s\quad (*)\\
=&  \alpha^r_{t_1+t_2}(x).
\end{align*}
If $\sgn t_1=  \sgn t_2$, the one but last line $(*)$ is obvious. To cover the
general case, it suffices then to show $(*)$  for $t_2=-t_1 <0$ (or analogously
$t_2=-t_1 >0$) since all other cases can be reduced to a composition of the
latter one and the case $\sgn t_1=  \sgn t_2$. 
To this end, notice that $(t_1/t_2)^{n_1}= (-1)^{n_1}$ is alternating, leading
to cancellation of mutually consecutive terms with fixed $n$; this leaves us
only with the term $n=0$. Thus in particular,
$\alpha^r_t\alpha^r_{-t}(x)=\alpha_0(x)=\alpha^r_0(x)$, for every $t\in\R$,
proving the group property for $(\alpha^r_t)_{t\in\R}$. 

Concerning continuity in $t$, we recall that both $a_r$ and $y$ lie in
$\Cci(\delta)_I$, by assumption. Thus term-by-term differentiation of the series
 
\[
\phi(x \alpha^r_t(y)z)= \sum_{n\in\NN} \rmi^n \int_{\Delta^t_n} 
\phi\Big(x \ad(\alpha_{s_1}(a_r))\cdots \ad(\alpha_{s_n}(a_r)) (\alpha_t(y))
z\Big) \rmd^n s 
\]
in $t=0$ yields a convergent series again with nonzero contribution only for the
zeroth and first summand, namely 
\begin{align*}
-\rmi\frac{\rmd}{\rmd t} \phi\big(x\alpha^r_t(y) z\big)\restriction_{t=0}
=& -\rmi\frac{\rmd}{\rmd t} \phi\big(x\alpha_t(y) z\big)\restriction_{t=0}
-\rmi\frac{\rmd}{\rmd t} \rmi \int_0^t \phi \big( x
\ad(\alpha_{s}(a_r))(\alpha_{t}(y)))z\big) \rmd s \restriction_{t=0} \\ 
=& \phi\big(x \delta^2(y) z\big)+ \phi \big( x (\ad(a_r)(y))z\big)
=  \phi\big(x \delta_r^2(y) z\big),
\end{align*}
making use of the weak supersymmetry property $(S_5)$ of
$((\alpha_t)_{t\in\R},\delta)$. We call $\delta_r^2=\delta^2+\ad(a_r)$ the
\emph{$\phi$-weak generator of $\alpha^r$}. 

Now let us turn to the difference $\alpha_t^r(x)-\alpha_t^q(x)$.
Remembering $a_r\in\Cci(\delta)_c$ and \eqref{eq:JLO-exp-est}, we have for the
$n$-th term in the sum the following upper bound: 
\begin{align*}
|t|^n \Big| \int_{\Delta_n} &
\ad(\alpha_{t s_1}(a_r))\cdots \ad(\alpha_{t s_n}(a_r))(x) 
- \ad(\alpha_{t s_1}(a_q))\cdots \ad(\alpha_{t s_n}(a_q))\alpha_t(x) \rmd^n s
\Big|\\ 
\le& |t|^n \int_{\Delta_n} \Big|\sum_{k=1}^n x \ad(\alpha_{t s_1}(a_r))\cdots
\ad(\alpha_{t s_{k}}(a_r)) \ad(\alpha_{t s_{k+1}}(a_q))\cdots \ad(\alpha_{t
s_n}(a_q))\alpha_t(x)\\ 
&- \ad(\alpha_{t s_1}(a_r))\cdots \ad(\alpha_{t s_{k-1}}(a_r))\ad(\alpha_{t
s_{k}}(a_q))\cdots \ad(\alpha_{t s_n}(a_q))\alpha_t(x) \Big| \rmd^n s \\ 
\le& \frac{|t|^n}{n!} \sum_{k=1}^n \|a_r\|^{k-1} \|a_r-a_q\| \|a_q\|^{n-k} \|x\|
\\ 
\le& \frac{|t|^n}{(n-1)!} (\|a_r\|^{n-1} + \|a_q\|^{n-1}) \|a_r-a_q\| \|x\|. 
\end{align*}
Summing over $n$ and using the power series expansion of the exponential
function and\linebreak $\|a_r-a_q\|\le 2|r-q| (\|\delta(Q)\|+\|Q^2\|)$, we
obtain the stated upper bound. 

Finally, we have to check that every $\alpha^r_t$ (but not $\gamma^r_t$) is
multiplicative. This will follow immediately from
multiplicativity of $\alpha_t$ and the subsequent Lemma
\ref{lem:JLO-gamma}(2)\&(3), which does not make use of multiplicativity, namely
\[
\alpha^r_t(xy)=\gamma^r_t(\unit)\alpha_t(x)\alpha_t(y)\gamma^r_t(\unit)^*=
\gamma^r_t(\unit)\alpha_t(x)\gamma^r_t(\unit)^*\gamma^r_t(\unit)\alpha_t(y)
\gamma^r_t(\unit)^*= \alpha^r_t(x)\alpha^r_t(y),
\]
for all $x,y\in\dom(\phi)_c$. We conclude that $\alpha^r_t$ is a *-automorphism.
\end{proof}

\begin{definition}\label{def:JLO-perturb}
Given a graded quantum dynamical system
$(\AA,\gamma,(\alpha_t)_{t\in\R},\delta)$ and an odd selfadjoint element
$Q\in\Cci(\delta)_c$, let
\[
 \delta_r := \delta + r[Q,\cdot],\quad r\in[0,1].
\]
Then $(\AA,\gamma,(\alpha^r_t)_{t\in\R},\delta_r)$, for every $r\in[0,1]$, is
called a \emph{perturbed graded quantum dynamical system} for
$(\AA,\gamma,(\alpha_t)_{t\in\R},\delta)$. If $\phi$ is an sKMS functional for
the original system, then the corresponding \emph{perturbed functional} is given
by
\[
 \phi^r(x) := \phi \big(x \gamma^r_{\rmi}(\unit)\big), \quad
x\in\dom(\phi^r):=\dom(\phi)_c.
\]
\end{definition}

Note first that $t\in\R\mapsto \gamma^r_t(\unit)$ need not be analytically
continuable, but the above expression is just a sloppy notation for the analytic
continuation of the composed function $t\mapsto \phi(x\gamma_t^r(\unit))$, which
will be proved in Proposition \ref{prop:JLO-phi} to be well-defined.
Second, for $r\not=0$, $\alpha^r$ loses its geometric interpretation: in
general,
\[
 \alpha^r_t(\A(I)) \not\subset \A(t+I), \quad t\in\R, I\in\I.
\]
Only for $I$ sufficiently large and $t$ small such that $Q\in\A(I_0)$ and
$I_0\subset I \cap (t+I)$, this inclusion still holds.

\begin{lemma}\label{lem:JLO-gamma}
In the above setting, we have the following equalities:
\begin{itemize}
\item[$(1)$] $\gamma^r_t(x)= \gamma^r_s(\unit) \alpha_s(\gamma_{t-s}^r(x))$, for
all $s,t\in\R$ and $x\in\dom(\phi)_c$.
\item[$(2)$] $\gamma^r_t(\unit)^*=\alpha_t(\gamma^r_{-t}(\unit))$ and 
$\gamma^r_t(\unit)\gamma^r_t(\unit)^* =\gamma^r_t(\unit)^*\gamma^r_t(\unit) =
\unit$, 
for all $t\in\R$.
\item[$(3)$] $\alpha^r_t(x)=\gamma^r_t(\unit)\alpha_t(x)\gamma^r_t(\unit)^*$,
for all $t\in\R$ and $x\in\dom(\phi)_c$.
\item[$(4)$] $\alpha^r_t(x)\gamma^r_t(y)=\gamma^r_t (xy)$, for all $t\in\R$ and
$x,y\in\dom(\phi)_c$.
\end{itemize}
\end{lemma}

\begin{proof}
(1) follows from the one-parameter group property and the definition of
$\gamma^r_t$:
\[
\gamma^r_t(x)= \gamma^r_s(\gamma^r_{t-s}(\unit))=
 \gamma^r_s(\unit) \alpha_s(\gamma_{t-s}^r(x)).
\]

(2) The first statement is obvious from the definition of $\gamma^r_t$ and the
self-adjointness of $a_r$, checked summand-wise. Combining it with statement
(1), we get
\[
\gamma^r_t(\unit)\gamma^r_t(\unit)^* = \gamma^r_t(\unit)
\alpha_t(\gamma_{-t}^r(\unit)) 
=\gamma^r_0(\unit)=\unit
\]
and
\[
\gamma^r_t(\unit)^* \gamma^r_t(\unit) = \alpha_t \big(\gamma_{-t}^r(\unit)
\alpha_{-t}(\gamma^r_t(\unit))\big) =\alpha_t (\gamma^r_0(\unit)) = \unit.
\]

(3) Considering the defining sum of
$\gamma^r_t(\unit)\alpha_t(x)\gamma^r_t(\unit)^*$, we have to compute
\[
\sum_{n=0}^\infty \sum_{m=0}^n 
(\rmi t)^n (-1)^{m} \int_{\Delta_{m}}\int_{\Delta_{n-m}} 
\alpha_{s_1t}(a_r)\cdots \alpha_{s_{n-m}t}(a_r) \alpha_t(x)
\alpha_{q_m t}(a_r)\cdots \alpha_{q_1 t}(a_r) \rmd^{n-m} s \rmd^m q.
\]
Now we need a bit of combinatoric thinking. Notice that, for fixed $m$ and every
$s\in\Delta_{n-m}$, $q\in\Delta_m$, $j\in\{1,...,m\}$, there is
$k_j\in\{0,...,n-m\}$ such that $s_{k_j}\le q_j \le s_{k_j+1}$, and inserting
the components of $q$ between those of $s$ according to this ordering, we obtain
an element $u\in\Delta_n$. Varying $s$ and $q$ with fixed such $k=(k_1,...,k_m)$
produces all $u\in\Delta_n$. Varying $k$ produces another copy of $\Delta_n$
and there are exactly ${n \choose m}$ ways (indexed by the tuples $k$). The
integral in the above sum over the summand with given $n$ and $m$ and varying
$k$ corresponds precisely to all the ${n \choose m}$ summands obtained by
writing out $\ad(y)(z)=yz-zy$ in
\[
(\rmi t)^n \int_{\Delta_n} 
\ad(\alpha_{u_1 t}(a_r))\cdots \ad(\alpha_{u_n t}(a_r)) (\alpha_t(x)) \rmd^n s, 
\]
with $m$ terms $\alpha_{u_it}(a_r)$ on the right and $n-m$ on the left of
$\alpha_t(x)$. Summing then over $m$ and $n$ concludes the proof.

(4) is now a direct consequence of (1), (2) and (3):
\[
\alpha^r_t(x)\gamma^r_t(y) 
=\gamma^r_t(\unit)\alpha_t(x) \gamma^r_t(\unit)^*\gamma^r_t(\unit) \alpha_t(y) =
\gamma^r_t(\unit)\alpha_t(xy)=\gamma^r_t(xy).
\]
\end{proof}

\begin{lemma}\label{lem:JLO-phi}
In the above setting, let $\phi$ be an sKMS functional for
$(\AA,\gamma,(\alpha_t)_{t\in\R},\delta)$.
\begin{itemize}
\item[$(1)$]  The function $t\in\R^{n+1}\mapsto \phi
\big(\alpha_{t_1}(x_1)\cdots \alpha_{t_n}(x_n) \alpha_{t_{n+1}}(x_0)\big) $ has
a unique analytic continuation to $\T^{n+1}$ and, for every $x_i\in\dom(\phi)_c$
and $z\in\T^n$, we have
\[
\phi \big(\alpha_{z_1}(x_1)\cdots \alpha_{z_n}(x_n) \alpha_{\rmi}(x_0)\big) 
=\phi \big(x_0 \alpha_{z_1}(\gamma(x_1))\cdots \alpha_{z_n}(\gamma(x_n)) \big).
\]
\item[$(2)$] For all $x_i\in\dom(\phi)_c$, the function
\[
z\in \T^n \mapsto \overline{\phi \big(\alpha_{\bar{z}_n}(x_n)\cdots
\alpha_{\bar{z}_1}(x_1) \big) }
\]
is analytic and we have
\[
\overline{\phi \big(\alpha_{\bar{z}_n}(x_n)\cdots \alpha_{\bar{z}_1}(x_1) \big)
} =
\phi \big(\alpha_{z_1}(x_1^*)\cdots \alpha_{z_n}(x_n^*)\big).
\]
\end{itemize}
\end{lemma}

\begin{proof}
(1) The unique analytic continuation has been obtained in
\eqref{eq:JLO-analcont}. Keeping the first $n$ variables real, the sKMS property
implies
\[
\phi \big(\alpha_{t_1}(x_1)\cdots \alpha_{t_n}(x_n) \alpha_{\rmi}(x_0)\big) 
=\phi \big(x_0 \gamma(\alpha_{t_1}(x_1)\cdots \alpha_{t_n}(x_n)) \big)
=\phi \big(x_0 \alpha_{t_1}(\gamma(x_1))\cdots \alpha_{t_n}(\gamma(x_n)) \big),
\]
for all $t\in\R^n$. From the uniqueness of the analytic continuation to $\T^n$
we obtain the statement.

(2) Let 
\[
G_{x_1,\ldots,x_n}(t_1,\ldots,t_n):= \phi\big(\alpha_{t_1}(x_1)
\alpha_{t_2}(x_2) \cdots \alpha_{t_n}(x_n)\big),
\]
for $t\in\R^n$. It has a unique analytic continuation to $\T^n$ according to
\eqref{eq:JLO-analcont}. Moreover, $\alpha$-invariance of $\phi$ shows that
$G_{x_1,\ldots,x_n}(t_1,\ldots,t_n)=G_{x_1,\ldots,x_n}(t_1+t,t_2+t,\ldots,
t_n+t)$, for all $t\in\R$, hence it actually extends uniquely to an analytic
function on $\T_1^{n-1}:=\{z\in\C^n:\; \Im z_j \le \Im z_{j+1}, j=1,\ldots
n-1,\; \Im z_n-\Im z_1 \le 1\}$. Notice that
$(\bar{z}_n,\ldots,\bar{z}_1)\in\T_1^{n-1}$ if $z\in \T_1^{n-1}$.

Since $\phi(x^*)=\overline{\phi(x)}$, we have
\[
G_{x_1^*,\ldots, x_n^*}(t_1,\ldots t_n) = \overline{
G_{x_n,\ldots,x_1}(t_n,\ldots, t_1)}= \overline{ G_{x_n,\ldots,x_1}}(t_n,\ldots,
t_1),
\]
for all $t\in\R^n$. Since the left-hand side has a unique analytic continuation
to $\T_1^{n-1}$, so must the right-hand side, namely
\[
G_{x_1^*,\ldots, x_n^*}(z_1,\ldots z_n) 
= \overline{ G_{x_n,\ldots,x_1}}(z_n,\ldots, z_1)
= \overline{ G_{x_n,\ldots,x_1}(\bar{z}_n,\ldots, \bar{z}_1)}.
\]
\end{proof}

With these tools at hand, let us study perturbations of sKMS functionals.

\begin{proposition}\label{prop:JLO-phi}
Suppose $\phi$ is a local-exponentially bounded sKMS functional for a graded
quantum dynamical system $(\AA,\gamma,(\alpha_t)_{t\in\R},\delta)$ and let
$Q\in\Cci(\delta)_c$ be an odd selfadjoint perturbation. For every
$r\in[0,1]$, the corresponding perturbed functional $(\phi^r,\dom(\phi^r))$ is
a well-defined sKMS functional with respect to the perturbed system
$(\AA,\gamma,(\alpha^r_t)_{t\in\R},\delta_r)$ in Definition
\ref{def:JLO-perturb}, but in general not satisfying the bounds in $(S_2)$ and
$(S_6)$, nor local normality.
\end{proposition}

\begin{proof}
Notice first that, for every $x,y,z\in\dom(\phi)_c$, the function
\begin{align*}
(t,u)\in\R^2 \mapsto& \phi(x \gamma^r_t(y) \alpha_u(z))
=  \sum_{n\in\NN}  (\rmi t)^n \int_{\Delta_n} 
\phi \big(x \alpha_{s_1t}(a_r)\cdots \alpha_{s_nt}(a_r) \alpha_t(y)\alpha_u(z)
\big) \rmd^n s, 
\end{align*}
has a unique analytic continuation to the tube $\T^2=\{(t,u)\in\C^2: 0\le
\Im(t)\le \Im(u) \le 1\}$ which is analytic on the interior of $\T^2$. This can
be seen as follows: Each of the summands on the right-hand side has a unique
analytic continuation to $\T^{2}$, which is essentially a consequence of
\eqref{eq:JLO-analcont}; moreover those continuations are bounded as in
\eqref{eq:JLO-exp-est}. Thus the sum of those continuations converges compactly
and defines an analytic continuation of $t\mapsto \phi(x \gamma^r_t(y)
\alpha_u(z))$ to $\T^2$ according to Weierstrass' convergence criterion. In the
same way but more generally, we see that, for every $n\in\N$ and
$x_i\in\dom(\phi)_c$,
\begin{equation}\label{eq:JLO-lem1}
(t_1,\ldots,t_n)\in\R^n \mapsto \phi \big(x_0 \gamma^r_{t_1}(x_1)
\alpha_{t_1}(\gamma^r_{t_2-t_1}(x_1))\cdots
\alpha_{t_{n-1}}(\gamma^r_{t_n-t_{n-1}}(x_n))\big)
\end{equation}
has a unique analytic continuation to $\T^{n}$.

Let us keep on record an explicit local bound for the analytic continuation.
Let $I$ be large enough so that $Q\in\dom(\phi)_I$. Then, for every $r\in[0,1]$
and $x\in\dom(\phi)_I$, we find 
\begin{equation}\label{eq:JLO-perturb-bound}
\begin{aligned}
|\phi(x\gamma^r_{\rmi} (\unit))| \le& 
\sum_{n=0}^\infty \Big| \int_{\Delta_n} 
\phi(x\alpha_{\rmi s_1}(a_r)\cdots  \alpha_{ \rmi s_n}(a_r)) \rmd^n s \Big| \\
\le& \sum_{n=0}^\infty \frac{1}{n!} C_1 \rme^{4C_2 (1+|I|)^2(n+1)} \| a_r\|^n
\|x\| \le C_I \|x\|,
\end{aligned}
\end{equation}
with $C_I := C_1\exp(4C_2 (1+|I|)^2+\|a_r\| \rme^{4C_2 (1+|I|)^2}) >0$ using
\eqref{eq:JLO-exp-est}. Thus we have local boundedness of $\phi^r$ for
sufficiently large interval $I$, hence for all intervals (by isotony). We
expect, however, neither local-exponential boundedness nor local normality for
$\phi^r$, so only a weaker but for our purposes sufficient version of $(S_1)$
and $(S_6)$.

We have to check the sKMS property $(S_2)$, and we may do this summand-wise
again owing to the above reasoning. Given $x,z\in\dom(\phi)_c$ and applying
Lemma \ref{lem:JLO-gamma}(4), we see that $\phi(x \alpha^r_t(z)
\gamma^r_{u}(\unit)))
=\phi(x \gamma^r_t(z) \alpha_t(\gamma^r_{u-t}(\unit)))$ has a unique analytic
continuation in $(t,u)$ to $\T^2$ according to \eqref{eq:JLO-lem1}. Then
\[
F_{x,z}(t):= \phi(x \alpha^r_t(z) \gamma^r_{\rmi}(\unit)), \quad t\in\R,
\]
has a unique analytic continuation to $\T^1$, and
\begin{equation}\label{eq:JLO-Fxz}
\begin{aligned}
F_{x,z}(t+\rmi) =& \phi\big(x \alpha^r_{t+\rmi}(z)\gamma^r_{\rmi}(\unit)\big) 
= \phi\big(x \alpha^r_{\rmi}(\alpha^r_t(z))\gamma^r_{\rmi}(\unit)\big) \\
=& \phi\big(x \gamma^r_{\rmi}(\unit) \alpha_{\rmi} (\alpha^r_{t}(z))\big) \\
=& \sum_{n=0}^\infty (-1)^n \int_{\Delta_n} \phi \big(x \alpha_{\rmi
s_1}(a_r)\cdots \alpha_{\rmi s_n}(a_r)\alpha_{\rmi} (\alpha^r_{t}(z))\big)
\rmd^n s\\
=& \sum_{n=0}^\infty (-1)^n \int_{\Delta_n} \phi \big(\alpha^r_{t}(z) \gamma(x)
\alpha_{\rmi s_1}(\gamma(a_r))\cdots \alpha_{\rmi s_n}(\gamma(a_r))\big) \rmd^n
s\\
=& \phi\big(\alpha^r_t(z) \gamma(x)\gamma^r_{\rmi}(\unit) \big).
\end{aligned}
\end{equation}
Here the second line follows from the analytic continuation of the function 
\[
s\mapsto \phi\big(x \alpha^r_{t+s}(z)\gamma^r_{s}(\unit) \big)
=\phi\big(x \alpha^r_{s}(\alpha^r_t(z))\gamma^r_{s}(\unit) \big)
=\phi\big(x \gamma^r_s(\unit)\alpha_s(\alpha^r_{t}(z)) \big)
\]
applying Lemma \ref{lem:JLO-gamma}(4); the fourth one follows from Lemma
\ref{lem:JLO-phi}(1), while the last line is clear since $\gamma(a_r)=a_r$. Thus
the sKMS property $(S_2)$ (except for the polynomial growth condition) holds for
$\phi^r$.

Let us check the remaining conditions for sKMS functionals. We may choose
$\dom(\phi^r):=\dom(\phi)_c$ since $\gamma^r_t(\unit)$ is smooth and localized.
Then, for all $x\in\dom(\phi)_c$,
\begin{align*}
\phi^r(x^*) 
=& \sum_{n=0}^\infty (-1)^n \int_{\Delta_n} \phi \big(x^* \alpha_{\rmi
s_1}(a_r)\cdots \alpha_{\rmi s_n}(a_r)\big) \rmd^n s\\
=& \sum_{n=0}^\infty (-1)^n \int_{\Delta_n} \phi \big(x^* \alpha_{\rmi
s_1}(\gamma(a_r))\cdots \alpha_{\rmi s_n}(\gamma(a_r))\big) \rmd^n s\\
=& \sum_{n=0}^\infty (-1)^n \int_{\Delta_n} \phi \big(\alpha_{\rmi
s_1}(a_r)\cdots \alpha_{\rmi s_n}(a_r)\alpha_{\rmi} (x^*)\big) \rmd^n s
\\
=& \overline{ \sum_{n=0}^\infty (-1)^n \int_{\Delta_n} \phi
\big(\alpha_{-\rmi}(x) \alpha_{-\rmi s_n}(a_r)\cdots \alpha_{-\rmi
s_1}(a_r)\big) \rmd^n s }\\
=& \overline{ \sum_{n=0}^\infty (-1)^n \int_{\Delta_n} \phi \big(x
\alpha_{\rmi-\rmi s_n}(a_r)\cdots \alpha_{\rmi-\rmi s_1}(a_r)\big) \rmd^n s } \\
=& \overline{ \sum_{n=0}^\infty (-1)^n \int_{\Delta_n} \phi \big(x \alpha_{\rmi
s_1'}(a_r)\cdots \alpha_{\rmi s_n'}(a_r)\big) \rmd^n s' } \\
=& \overline{ \phi\big(x \gamma^r_{\rmi}(\unit) \big)} = \overline{\phi^r(x)}.
\end{align*}
Here the third line follows from Lemma \ref{lem:JLO-phi}(1), the fourth one from
Lemma \ref{lem:JLO-phi}(2), and the fifth one from (constant) analytic
continuation to $\C$ of the constant function $t\in\R \mapsto \phi \circ
\alpha_t(y)$, \ie from $\phi\circ\alpha_{-\rmi}(y) = \phi(y)$. We conclude with
a change of variable $s_k'=1-s_{n+1-k}$. Thus Hermitianity and all other
properties in $(S_0)$ are clear.

The normalization property $\phi^r(\unit)=1$ will be shown in
\eqref{eq:JLO-phi-r1} as a corollary of the proof of Theorem \ref{th:JLO-main},
which does not make use of $(S_3)$ but instead only of the finiteness of
$\phi^r(\unit)$.

Concerning $(S_4)$, we have to show $\phi^r\circ\delta_r(z)=0$, for
$z\in\Cci(\delta_r)_c$. 
We claim that
\begin{equation}\label{eq:JLO-lem2}
\phi(z e(t))=0,\quad e(t):= \delta(\gamma^r_t(\unit))+ rQ \gamma^r_t(\unit) -
r\gamma^r_t(\unit) \alpha_t(Q),
\end{equation}
for all $t\in\R$, which implies in particular that $t\in\R\mapsto\phi(z e(t))$
has a unique and trivial analytic continuation to $\C$. Using then first
$\phi\circ\delta=0$ and analytic continuation for the first
term on the right-hand side below and $\phi\circ\gamma=\gamma$ and Lemma
\ref{lem:JLO-phi}(1) with a similar reasoning as in \eqref{eq:JLO-Fxz} for the
second one, we obtain
\begin{align*}
 \phi^r\circ\delta_r(z) =& \phi\big(\delta(z)\gamma^r_{\rmi}(\unit)\big)
+ r\phi\big(Qz \gamma^r_{\rmi}(\unit)\big) - r
\phi\big(\gamma(z)Q\gamma^r_{\rmi}(\unit)\big)\\
=& \phi\big(-\gamma(z)\delta(\gamma^r_{\rmi}(\unit))\big) +r \phi\big(\gamma(z)
\gamma^r_{\rmi}(\unit)\alpha_{\rmi}(Q) \big) - r
\phi\big(\gamma(z)Q\gamma^r_{\rmi}(\unit)\big) \\
=& - \phi \big(\gamma(z) e(\rmi)\big) = 0.
\end{align*}

\emph{Proof of the claim.} According to \eqref{eq:JLO-lem1}, $E: t\in\R\mapsto
\phi(z e(t))$ extends to an analytic function on $\T^1$; moreover, it is
differentiable on $\R$ as follows from the definition of $\gamma^r_t$
in Proposition \ref{prop:JLO-alpha}. Differentiation by $t$ together with
properties $(S_4)$ and $(S_5)$ for $\phi$ yields
\begin{align*}
-\rmi\frac{\rmd}{\rmd t} E(t)
=& \phi\big( z (\delta +rQ)(\delta^2+ a_r)(\gamma^r_t(\unit)) - 
z (\delta^2 + a_r)( \gamma^r_t(\unit)\alpha_t(rQ)) \big)\\ 
=& \phi \big( z(\delta^2+ a_r) e(t)\big) \\
=& \phi \big(\delta^2 (z e(t))\big) - \phi\big( \delta^2(z) e(t) \big) + \phi
\big( z a_r e(t) \big)\\
=& \phi\big( (\delta^2+a_r)(z^*)^* e(t)\big), \quad t\in\R.
\end{align*}
Recursively one finds $z_{r,n}:=(\delta^2+ a_r)^n(z^*)^* \in\Cci(\delta)_c$ such
that
\[
(-\rmi)^n \frac{\rmd^n}{\rmd t^n} E(t) =\phi \big( z_{r,n} e(t) \big) , \quad
n\in\NN, t\in\R.
\]
Since all $z_{r,n} e(t)$ are localized (uniformly for $t$ in bounded intervals)
and $\phi$ is locally bounded, all derivatives of $E$ are continuous; in other
words, $E$ is smooth on $\R$ and furthermore, according to our preceding
discussion, analytic on the interior of $\T^1$. Since $e(0)=0$, we get for all
derivatives: $E^{(n)}(0)=0$, $n\in\N_0$. Applying the $\Cci$-version of Schwarz'
reflection principle \cite[Th.1]{BL90} shows that $E\equiv 0$ on the whole strip
$\T^1$, thus \eqref{eq:JLO-lem2}, which proves the claim.

Finally, $(S_5)$ is shown using the above analytic continuation properties
together with the expression for the $\phi$-weak generator of $\alpha^r$ in
Proposition \ref{prop:JLO-alpha}: 
\begin{align*}
-\rmi \frac{\rmd}{\rmd t}  \phi^r \big( x\alpha^r_t(y) z)\restriction_{t=0}
=& -\rmi  \frac{\rmd}{\rmd t} \phi \big( x\alpha^r_t(y)
z\gamma^r_{\rmi}(\unit)\big)\restriction_{t=0}
=  \phi \big( x\delta_r^2(y) z\gamma^r_{\rmi}(\unit)\big)\restriction_{t=0}
= \phi^r \big( x \delta_r^2(y) z\big).
\end{align*}
\end{proof}

We are now ready for the main result:

\begin{theorem}\label{th:JLO-main}
Given a local-exponentially bounded sKMS functional $\phi$ for a graded quantum
dynamical system $(\AA,\gamma,(\alpha_t)_{t\in\R},\delta)$, the even JLO cochain
$\tau$ over $\Cci(\delta)^\gamma_c$ is a local-entire cyclic cocycle. Moreover,
it is homotopy-invariant: given an odd selfadjoint perturbation $Q\in
\Cci(\delta)_c$, the corresponding perturbed functionals $\phi^r$ for the
perturbed system $(\AA,\gamma,(\alpha^r_t)_{t\in\R},\delta_r)$ in Definition
\ref{def:JLO-perturb} give rise to even JLO local-entire cyclic cocycles
$\tau^r$ again, which are mutually cohomologous, for all $r\in[0,1]$. 
\end{theorem}

\begin{proof}
The first statement is just Theorem \ref{prop:JLO-cocycle}. 

Let us consider the perturbed functionals, and let $I\in\I$ be an arbitrary
fixed interval such that $Q\in\Cci(\delta)_I$. Then by definition of
$\gamma_t^r$
and $\alpha^r_t=\gamma^r_t(\unit)\alpha_t(\cdot)(\gamma^r_t(\unit))^*$, we have,
for every $s\in\R^{n+1}$ and $x_i\in\Cci(\delta)_c$:
\begin{equation}\label{eq:JLO-phirsum}
\begin{aligned}
\phi( x_0 & \alpha^r_{s_1}(x_1) \cdots \alpha^r_{s_n}(x_n)
\gamma^r_{s_{n+1}}(\unit)) \\
= & \sum_{\vec{k}\in\NN^{n+1}} \rmi ^{|\vec{k}|}
\int_{\Delta^{s_1}_{k_1}} \cdots \int_{\Delta^{s_{n+1}-s_n}_{k_{n+1}}}
\phi\Big( x_0 \ad( \alpha_{t_{1,1}}(a_r))\cdots
\ad(\alpha_{t_{1,k_1}}(a_r))\ad(\alpha_{s_1}(x_1))\cdots\\ 
& \cdots\ad(\alpha_{s_n+t_{n+1,1}}(a_r)) \cdots
\ad(\alpha_{s_n+t_{n+1,k_{n+1}}}(a_r))\alpha_{s_{n+1}}(\unit)\Big) 
\rmd^{k_{n+1}} t_{n+1}\cdots\rmd^{k_1}t_1
\end{aligned}
\end{equation}
Following the argument of \eqref{eq:JLO-analcont} or \eqref{eq:JLO-lem1}, each
of the integrands has a unique analytic continuation to $\T^{|\vec{k}|}$ (with
the usual multi-index notation), and the integrals thus have a continuation in
$s$ to $\T^{n+1}$. Moreover, according to \eqref{eq:JLO-exp-est} and as
explained also in the proof of Proposition \ref{prop:JLO-phi}, the latter
are bounded by
\[
C_1 \rme^{2C_2 (1+|I|)^2(|\vec{k}|+1)} 
\rme ^{C_2 (1+|I|) \sum_{i=1}^{n+1} k_i |z_i|^2} \frac{1}{\vec{k}!}
|z|^{\vec{k}} \; \|2 a_r\|^{|\vec{k}|} \cdot \|x_0\|\cdots \| x_n\|, \quad
z\in\T^{n+1}.
\]
Thus \eqref{eq:JLO-phirsum} is the sum of analytically continuable functions and
the sum of the continuations is compactly convergent and hence analytic. In
fact, we have for $z\in\T^{n+1}$:
\begin{align*}
|\phi( x_0 & \alpha^r_{z_1}(x_1)  ...\alpha^r_{z_n}(x_n)
\gamma^r_{z_{n+1}}(\unit)) |
\\
\le & \sum_{\vec{k}\in\NN^{n+1}} C_1 \rme^{2C_2 (1+|I|)^2(|\vec{k}|+1)} 
\rme ^{C_2 (1+|I|) \sum_{i=1}^{n+1} k_i |z_i|^2} \frac{1}{\vec{k}!}
|z|^{\vec{k}} \|2a_r\|^{|\vec{k}|} \cdot \|x_0\|\cdots \| x_n\| \\
\le& C_1 \rme^{ 2C_2(1+|I|)^2} \exp \Big(\Big( |z_1|\rme^{2C_2(1+ |I|)^2
(|z_1|^2+1)} +...+ |z_{n+1}|\rme^{2C_2(1+ |I|)^2 (|z_{n+1}|^2+1)} \Big)
\|2a_r\| \Big) \\
&\times \|x_0\|\cdots \| x_n\| \\
\le& C_1  \rme^{ 2C_2(1+|I|)^2}\exp \Big((n+1)\max_i(|z_i|)
\rme^{2C_2 (1+|I|)^2 (\max_i |z_i|^2 +1)} \|2a_r\| \Big) \|x_0\|\cdots \| x_n\|.
\end{align*}
Integration over $\rmi\Delta_n$ then gives rise to the
well-defined 
\[
F^r_n(x_0,...,x_n):=  \int_{\Delta_n} \phi^r(x_0 \alpha^r_{\rmi s_1}(x_1)\cdots
\alpha^r_{\rmi s_n}(x_n)) \rmd^n s, \quad x_i\in \Cci(\delta)_c, 
\]
and for the corresponding JLO cochain $\tau^r$ we therefore find, for every
$I\in\I$: 
\begin{align*}
\sqrt{n} \|\tau_n^r\restriction_{\Cci(\delta)^\gamma_I} \|_*^{1/n}
\le& \sqrt{n} \Big( \frac{1}{n!}C_1  \rme^{ 2C_2(1+|I|)^2}\exp\big((n+1)
\rme^{4C_2 (1+|I|)^2} \|2a_r\| \big) \Big)^{1/n}\\
\sim& \frac{1}{\sqrt{n}} \exp \Big( \rme^{4C_2 (1+|I|)^2}
\|2a_r\|\Big), \quad n\ra\infty,
\end{align*}
which converges to $0$ for $n\ra\infty$, so $\tau^r$ is in fact
local-entire. The cyclic cocycle condition is purely algebraic and literally
goes like (Part 2) of the proof of Theorem \ref{prop:JLO-cocycle}, based on the
fact that $\phi^r$ satisfies the sKMS condition for the perturbed system
$(\AA,\gamma,(\alpha^r_t)_{t\in\R},\delta_r)$. The precise growth factor in
$(S_2)$, which is different for the perturbed functional, does not play a role
here; it is needed in order to obtain the analytic continuations
\eqref{eq:JLO-analcont} and the bound \eqref{eq:JLO-exp-est}; in the case of the
perturbed system, we obtain continuation and upper bound from the corresponding
properties of the original system as just done. This way, Lemma
\ref{lem:JLO-cocycle}, reformulated for arbitrary $r\in[0,1]$, implies the
following equalities for $x_i\in\Cci(\delta)_c$ and $k=1,\ldots,n-1$:
\begin{gather}
F^r_n(x_0,...,x_n) = F^r_n(\gamma(x_n),x_0,...,x_{n-1}) \label{eq:JLO-F1}\\
F^r_{n}(x_0,x_1,...,\delta_r^2(x_k),...,x_n)=
F^r_{n-1}(x_0,...,x_{k-1}x_{k},...,x_n)-F^r_{n-1}(x_0,...,x_kx_{k+1},...,x_n) 
\label{eq:JLO-F2}\\
F^r_{n}(x_0,x_1,...,x_{n-1},\delta_r^2(x_n))=
F^r_{n-1}(x_0,...,x_{n-1}x_n)-F^r_{n-1}(\gamma(x_n)x_0,x_1,...,x_{n-1}) 
\label{eq:JLO-F4}\\
\sum_{j=0}^{n} F^r_{n+1}(\unit,x_{j},...,x_{n},\gamma(x_0),...,\gamma(x_{j-1}))
= F^r_{n}(x_0,...,x_{n}). \label{eq:JLO-F5}
\end{gather}
Moreover, together with \eqref{eq:JLO-BGLem} we obtain
\begin{equation}\label{eq:JLO-F6}
\sum_{j=0}^{n} F^r_n(\gamma(x_0),...,\gamma(x_{j-1}),\delta_r(x_j),
x_{j+1},...,x_n)=0.
\end{equation}

Concerning perturbation invariance of the cyclic cocycle $\tau$, we would
like to show that
\begin{equation}\label{eq:JLO-boundary}
 G^r_{n-1}(x_0,...,x_{n-1}) :=  \sum_{k=0}^{n-1} (-1)^k  F^r_n (x_0,
\delta_r(x_1),...,\delta_r(x_k),Q,..., \delta_r(x_{n-1})), \quad
x_i\in\Cci(\delta)^\gamma_c,
\end{equation}
for even and $G^r_{n-1}=0$ for odd $n\in\N$, defines a local-entire cochain on
$\Cci(\delta_r)^\gamma_c=\Cci(\delta)^\gamma_c$ such that
\begin{equation}\label{eq:JLO-main}
 \frac{\rmd}{\rmd r}\tau^{r} = \partial G^r.
\end{equation}
This would imply that, for every $q,r\in [0,1]$, the cochains $\tau^q$ and
$\tau^r$ differ by a coboundary, \ie are cohomologous. 

We first notice that the
cochain $(G^r_n)_{n\in \NN}$ above is clearly well-defined. The
local-entireness condition is verified in the same way as for the JLO cochain
$(\tau^r_n)_{n\in \NN}$  above, which becomes clear when writing $\tau^r$ in
terms of $F_n^r$, with $n\in\NN$. 

In order to prove \eqref{eq:JLO-main}, we have to calculate $\partial G^r$.
Applying \eqref{eq:JLO-F1}-\eqref{eq:JLO-F5} to the definition of the operator
$B$ in Definition \ref{def:JLO-ECC}, we obtain, for
$x_i\in\Cci(\delta)^\gamma_c$ (thus $\gamma(x_i)=x_i$,
$\gamma(\delta_r(x_i))=-\delta_r(x_i)$ and $\gamma(Q)=-Q$) and $n\in 2\NN$: 
\begin{align*}
BG^r_{n+1}(x_0,...,x_n) 
=& \sum_{j=0}^n  G^r_{n+1}(\unit,x_j,...,x_{j-1}) \\
=& \sum_{j=0}^n \Big( \sum_{k=0}^{j-1} (-1)^{k+2-j}
F^r_{n+1}(\unit,\delta_r(x_j),...,\delta_r(x_k),Q,...\delta_r(x_{j-1})) \\
&+ \sum_{k=j}^n (-1)^{k+1-j}
F^r_{n+1}(\unit,\delta_r(x_j),...,\delta_r(x_k),Q,...\delta_r(x_{j-1})) \Big)\\
=& \sum_{k=0}^n (-1)^{k+1} \Big(\sum_{j=0}^k (-1)^{j}
F^r_{n+1}(\unit,\delta_r(x_j),...,\delta_r(x_k),Q,...\delta_r(x_{j-1}))\\
&+\sum_{j=k+1}^n (-1)^{j+1}
F^r_{n+1}(\unit,\delta_r(x_j),...,\delta_r(x_k),Q,...\delta_r(x_{j-1})) \Big) \\
=& - \sum_{k=0}^{n} (-1)^k F^r_{n}(\delta_r(x_0),...,\delta_r(x_k), Q,
...,\delta_r(x_{n})),
\end{align*}
using \eqref{eq:JLO-F5} in the last line together with the fact that all
$\delta_r(x_i)$ and $Q$ are homogeneously odd. In the case of $b$ we find:
\begin{align*}
bG^r_{n-1} & (x_0,..,x_n)
= \sum_{j=0}^{n-1} (-1)^{j} G^r_{n-1}(x_0,...,x_jx_{j+1},...,x_{n}) 
+ (-1)^n G^r_{n-1}(x_nx_0,x_1,...,x_{n-1})\\
=& \sum_{j=0}^{n-1} \Big( \sum_{k=0}^{j-1} (-1)^{j+k}
F^r_{n}(x_0,\delta_r(x_1)...,\delta_r(x_k),Q,...,\delta_r(x_jx_{j+1}),...,
\delta_r(x_{n}))\\
&+ \sum_{k=j+1}^n (-1)^{j+k-1}
F^r_{n}(x_0,\delta_r(x_1),...,\delta_r(x_jx_{j+1}),...,\delta_r(x_k),Q,...,
\delta_r(x_{n})) \Big) \\
&+ \sum_{k=0}^{n-1} (-1)^{n+k}
F^r_{n}(x_nx_0,\delta_r(x_1),...,\delta_r(x_k),Q,...,\delta_r(x_{n-1}))
\end{align*}
\begin{align*}
=& \sum_{j=0}^{n-1} \Big( \sum_{k=0}^{j-1} (-1)^{j+k}
F^r_{n}(x_0,\delta_r(x_1)...,\delta_r(x_k),Q,...,\delta_r(x_j)x_{j+1} + x_j
\delta_r(x_{j+1}),...,\delta_r(x_{n}))\\
&+ \sum_{k=j+1}^n (-1)^{j+k-1} F^r_{n}(x_0,\delta_r(x_1),...,\delta_r(x_j)
x_{j+1} + x_j \delta_r(x_{j+1}),...,\delta_r(x_k),Q,...,\delta_r(x_{n})) \Big)
\\
&+ \sum_{k=0}^{n-1} (-1)^{n+k}
F^r_{n}(x_nx_0,\delta_r(x_1),...,\delta_r(x_k),Q,...,\delta_r(x_{n-1}))\\
=& \sum_{k=0}^n \Big( -\sum_{j=0}^k (-1)^{j+k-1}
F^r_{n+1}(x_0,\delta_r(x_1),...,\delta_r^2(x_j),...,\delta_r(x_k),Q,...,
\delta_r(x_{n})) \\
&-\sum_{j=k+1}^{n} (-1)^{j+k}
F^r_{n+1}(x_0,\delta_r(x_1)...,\delta_r(x_k),Q,...,\delta_r^2(x_j),...,
\delta_r(x_{n}))\Big)
\\
&+ \sum_{k=1}^n (-1)^{2k-1} F^r_{n}(x_0,\delta_r(x_1)...,\delta_r(x_{k-1}),x_k
Q,...,\delta_r(x_{n}))\\
&-\sum_{k=1}^n (-1)^{2k-1} F^r_{n}(x_0,\delta_r(x_1)...,\delta_r(x_{k-1}),Q
x_k,...,\delta_r(x_{n}))\\
=& \sum_{k=0}^n \Big( -\sum_{j=0}^k (-1)^{j+k-1}
F^r_{n+1}(x_0,\delta_r(x_1),...,\delta_r^2(x_j),...,\delta_r(x_k),Q,...,
\delta_r(x_{n})) \\
&-\sum_{j=k+1}^{n} (-1)^{j+k}
F^r_{n+1}(x_0,\delta_r(x_1)...,\delta_r(x_k),Q,...,\delta_r^2(x_j),...,
\delta_r(x_{n}))\Big) \\
&+ \sum_{k=1}^n
F^r_{n}(x_0,\delta_r(x_1)...,\delta_r(x_{k-1}),[Q,x_k],...,\delta_r(x_{n})) \\
=& \sum_{k=0}^{n} (-1)^k
F^r_{n+1}(\delta_r(x_0),\delta_r(x_1),...,\delta_r(x_k), Q, ...,\delta_r(x_{n}))
\\
&+ \sum_{k=0}^{n} (-1)^{k+k} F^r_{n+1}(x_0,\delta_r(x_1),...,\delta_r(x_k),
\delta_r(Q), ...,\delta_r(x_{n}))\\
&+ \sum_{k=1}^{n} F^r_{n}(x_0,\delta_r(x_1),...,[Q,x_k], ...,\delta_r(x_{n})),
\end{align*}
using \eqref{eq:JLO-F6} in the last and \eqref{eq:JLO-F2}, \eqref{eq:JLO-F4} in
the second but last equality. 

Let us turn to the left-hand side of \eqref{eq:JLO-main}: from the definition of
$\alpha^r_t$ and $\gamma^r_t$, we see that the functions $r\mapsto 
\phi(x\alpha^{r}_t (y)z)$ and $r\mapsto \phi(x\gamma^{r}_t(\unit))$ are
differentiable and that
\begin{align*}
\frac{\rmd}{\rmd r} \phi\big(x\alpha^{r}_t & (y)z\big) =
\sum_{n\in\NN} (\rmi t)^n \frac{\rmd}{\rmd r} \int_{\Delta_n} 
\phi\big(x\ad(\alpha_{t p_1}(a_r))\cdots \ad(\alpha_{t p_n}(a_r))
\alpha_t(y)z\big) \rmd^n p
\\
=&  \sum_{n\in\NN} \sum_{k=0}^n (\rmi s)^k (\rmi t-\rmi s)^{n-k} \int_0^t
\int_{\Delta_k} \int_{\Delta_{n-k}} 
\phi\big(x\ad(\alpha_{s p_1}(a_r))\cdots \ad(\alpha_{s p_k}(a_r)) \\
& \times \alpha_s \left( \ad(\dot{a}_r) \ad(\alpha_{(t-s)p_{k+1}}(a_r))\cdots
\ad(\alpha_{(t-s)p_n}(a_r))\alpha_{t-s}(y)\right)z\big) \rmd^{n-k} p \rmd^k p
\rmd s\\
=& \int_0^t \phi\big(x\alpha_s^r([\dot{a}_r, \alpha^r_{t-s}(y)])z\big) \rmd s
= \int_0^t \phi\big(x[\alpha_s^r(\delta_r(Q)), \alpha^r_t(y)]z\big) \rmd s,
\end{align*}
using $\dot{a}_r=\delta_r(Q)$. Analogously,
\begin{align*}
\frac{\rmd}{\rmd r} \phi\big(x\gamma^{r}_t(\unit)\big)
=&  \sum_{n\in\NN} \sum_{k=0}^n (\rmi s)^k (\rmi t-\rmi s)^{n-k} \int_0^t
\int_{\Delta_k} \int_{\Delta_{n-k}} 
\phi\big(x\alpha_{s p_1}(a_r)\cdots \alpha_{s p_k}(a_r) \\
& \times \alpha_s \left( \dot{a}_r \alpha_{(t-s)p_{k+1}}(a_r)\cdots
\alpha_{(t-s)p_n}(a_r)\alpha_{t-s}(x)\right)\big) \rmd^{n-k} p \rmd^k p \rmd s\\
=& \int_0^t \phi\big(x\gamma_s^r(\unit) \alpha_s(\dot{a}_r)
\alpha_s(\gamma^r_{t-s}(\unit))\big) \rmd s\\
=& \int_0^t \phi\big(x\gamma_s^r(\unit) \alpha_s(\dot{a}_r)
\gamma^r_{s}(\unit)^* \gamma^r_t(\unit)\big) \rmd s
= \int_0^t \phi\big(x\alpha^r_s(\delta_r(Q))\gamma^r_t(\unit)\big) \rmd s ,
\end{align*}
using the Lemma \ref{lem:JLO-gamma} in the last two steps.

Consider now, for given $x_0,..,x_n\in\Cci(\delta)^\gamma_c$, the functions
\[
r\in[0,1] \mapsto K_{x_0,..,x_n}(r; t_1,...,t_n,u) = \frac{\rmd}{\rmd r} \phi(
x_0 \alpha^r_{t_1}(\delta_r(x_1)) \cdots  \alpha^r_{t_n}(\delta_r(x_n))
\gamma^r_u(\unit)),
\quad t_i,u\in\R.
\]
Then we have, for $t\in\Delta_n^u$:
\begin{align*}
K_{x_0,..,x_n}& (r; t_1,...,t_n,u) \\
=& \sum_{j=1}^n  \phi\big( x_0 \alpha^r_{t_1}(\delta_r(x_1))
...\alpha^r_{t_j}([Q,x_j])\cdots  \alpha^r_{t_n}(\delta_r(x_n))
\gamma^r_u(\unit)\big) \\
&+ \sum_{j=1}^n  \int_0^{t_j} \phi\big( x_0 \alpha^r_{t_1}(\delta_r(x_1)) \cdots
 [\alpha^r_{s}(\delta_r(Q)),\alpha^r_{t_j}(\delta_r(x_j))]\cdots 
\alpha^r_{t_n}(\delta_r(x_n)) \gamma^r_u(\unit)\big) \rmd s \\
&+ \int_0^u \phi\big( x_0 \alpha^r_{t_1}(\delta_r(x_1)) \cdots 
\alpha^r_{t_n}(\delta_r(x_n)) \alpha_s(\delta_r(Q)) \gamma^r_u(\unit)\big) \rmd
s.
\end{align*}
In the same way as \eqref{eq:JLO-phirsum}, this has a unique analytic
continuation to the tube $\T^{n+1}$, and we obtain, for $t\in\Delta_n$:
\begin{align*}
K_{x_0,..,x_n}& (r; \rmi t_1,...,\rmi t_n,\rmi) 
= \sum_{j=1}^n  \phi\big( x_0 \alpha^r_{\rmi t_1}(\delta_r(x_1)) \cdots
\alpha^r_{\rmi t_j}([Q,x_j])\cdots  \alpha^r_{\rmi t_n}(\delta_r(x_n))
\gamma^r_{\rmi}(\unit)\big) + \\
& + \sum_{j=0}^n  \int_{t_j}^{t_{j+1}} \phi\big( x_0 \alpha^r_{\rmi
t_1}(\delta_r(x_1)) \cdots  \alpha^r_{\rmi t_j}(\delta_r(x_j)) \alpha^r_{\rmi
s}(\delta_r(Q)) \cdots  \alpha^r_{\rmi t_n}(\delta_r(x_n))
\gamma^r_{\rmi}(\unit)\big) \rmd s,
\end{align*}
where $t_0=0$ and $t_{n+1}=1$. Together with the definition of $\tau^r$ we thus
have
\begin{align*}
\frac{\rmd}{\rmd r} \tau^r_n(x_0,...,x_n) 
=& \int_{\Delta_n} K_{x_0,..,x_n}(r; \rmi t_1,..., \rmi t_n,\rmi) \rmd^n t\\
=&\sum_{k=1}^n  F^r_n(x_0,\delta_r(x_1),...,[Q,x_k],...,\delta_r(x_n)) \\
& + \sum_{k=0}^n  F^r_{n+1}( x_0,\delta_r(x_1),..., \delta_r(x_k),\delta_r(Q),
...,\delta_r(x_n))\\
=& (B G^r_{n+1} + b G^r_{n-1})(x_0,...,x_n) = (\partial G^r)_n(x_0,...,x_n).
\end{align*}
\end{proof}

As a corollary of the proof we find
\begin{equation}\label{eq:JLO-phi-r1}
\phi^r(\unit)= \tau^r_0(\unit) = \tau_0 (\unit) + \int_0^r (\partial G^q)_0
(\unit) \rmd q = \tau_0(\unit) + 0 = \phi(\unit) = 1, 
\end{equation}
which completes the proof of property $(S_3)$ in Proposition \ref{prop:JLO-phi}.

For the reader familiar with \cite{Hil1} or at least standard examples of
superconformal nets and the corresponding notation, we provide two short
illustrations of the above results:

\begin{example}\label{ex:JLO-FF}
\emph{Supersymmetric free field net}. The supersymmetric free field net $\A$ and
an associated sKMS functional $(\phi,\dom(\phi))$ have been extensively studied
in \cite[Sec.3]{Hil1}, and we refer to the notation introduced there, in
particular the construction of the superderivation $\delta$ and the sKMS
functional; $J$ and $F$ stand for the corresponding bosonic and fermionic free
field currents.

The corresponding even JLO cocycle $\tau$ on $\Cci(\delta)^\gamma_c$ is
nontrivial because
\[
\tau(\unit) = \tau_0(\unit) = \phi(\unit) = 1, 
\]
owing to the normalization condition $(S_3)$ on $\phi$.

An example of an admissible perturbation $\delta_r=\delta+r \ad Q$ which leaves
the class of $\tau$ invariant is given by
\[
 Q= \int_\R \alpha_t \left( 
 (J(f)-\rmi)^{-1}F(f) (J(f)+\rmi)^{-1} \right) h(t) \rmd t, 
\]
with arbitrary but fixed $f,h\in \Cci_c(\R)$, \ie compactly supported
$\R$-valued smooth functions on $\R$. This perturbation is selfadjoint, odd,
localized and smooth, \ie $Q\in \Cci(\delta)_c$.

It seems interesting to study explicit (co-)homology classes and pairings with
$K_0$-theory (and corresponding projections) and understand their physical
meaning. However, the computations are very tedious and left for future study.
The conceptually interesting projections investigated
in \cite[Sec.5]{CHL12} or \cite[Sec.4]{CCHW} are unfortunately global, hence not
in $\AA$ and not applicable here.

Obviously, we may replace the free field net by its rational extension as
discussed in \cite[Th.3.8]{Hil1}, and all the above results concerning the JLO
cocycle and perturbations should extend to that setting.
\end{example}

\begin{example}\label{ex:JLO-SVir}
\emph{Super-Virasoro net}. Since the super-Virasoro net with central charge
$c\ge 3/2$ is a subnet of the free field net, this example becomes a consequence
of the preceding one by restricting the cocycle to $\AA_{\SVir}\cap
\Cci(\delta)_c$, denoted by $\tau_{\SVir}$. Note that also the perturbed cocycle
defined by the perturbation $Q$  in the preceding example (probably not in
$\AA_{\SVir}\cap \Cci(\delta)_c$) restricts to a local-entire cyclic cocycle on
$\AA_{\SVir}\cap \Cci(\delta)_c$, which is cohomologous to $\tau_{\SVir}$.
Whether $\tau_{\SVir}$ is really meaningful and whether there are possible
perturbations \emph{in} $\AA_{\SVir}\cap \Cci(\delta)_c$ depends on whether 
$\A_{\SVir}(I)\cap \Cci(\delta)_c \subset \A_{\SVir}(I)$ is actually larger than
$\C\unit$. We expect this to be true, but a proof is missing so far, \cf
\cite[Th.3.10]{Hil1}.
\end{example}

\bigskip

{\footnotesize 

\noindent\textbf{Acknowledgements.}
I would like to thank Paolo Camassa, Sebastiano Carpi, and Roberto Longo for
several helpful and pleasant discussions, for comments and for pointing out
problems.
}


\bigskip

\begin{thebibliography}{xx}

\bibitem{BL90}
S.~Bell and L.~Lempert.
\newblock {A $\Cci$ Schwarz reflection principle in one and several complex
variables}
\newblock {\em J. Diff. Geom.} 32, 899--915, 1990.

\bibitem{BR2}
O.~Bratteli and D.~Robinson.
\newblock {\em Operator algebras and quantum statistical mechanics}. 
\newblock Springer, 1997.

\bibitem{BB}
J.~Bros and D.~Buchholz. Towards a relativistic KMS-condition.
\newblock {\em Nucl. Phys. B} 429, 291--318, 1994.

\bibitem{BG}
D.~Buchholz and H.~Grundling.
\newblock {Algebraic supersymmetry: a case study}.
\newblock {\em Commun. Math. Phys.} 272, 699--750, 2007.

\bibitem{BJ}
D.~Buchholz and P.~Junglas.
\newblock {On the existence of equilibrium states in local quantum field
theory}. 
\newblock {\em Commun. Math. Phys.} 121, 255–270, 1989.

\bibitem{BL00}
D.~Buchholz and R.~Longo.
\newblock Graded KMS functionals and the breakdown of supersymmetry.
\newblock {\em Adv. Theor. Math. Phys.}  3, 615--626, 2000.
\newblock Addendum: {\em Adv. Theor. Math. Phys.}  6, 1909--1910, 2000.

\bibitem{CLTW1}
P.~Camassa, R.~Longo, Y.~Tanimoto, and M.~Weiner.
\newblock {Thermal states in conformal QFT. I}.
\newblock {\em Commun. Math. Phys.} 309, 703--735, 2011.

\bibitem{CLTW2}
P.~Camassa, R.~Longo, Y.~Tanimoto, and M.~Weiner.
\newblock {Thermal states in conformal QFT. II}.
\newblock {\em Commun. Math. Phys.}, 315, 771--802, 2012.

%
\bibitem{CCHW}
S.~Carpi, R.~Conti, R.~Hillier, and M.~Weiner.
\newblock {Representations of conformal nets, universal C*-algebras and
K-theory}.
\newblock {\em Commun. Math. Phys.} 320, 275--300, 2013.

\bibitem{CHL12}
S.~Carpi, R.~Hillier, and R.~Longo.
\newblock {Superconformal nets and noncommutative geometry}.
\newblock {\em arXiv:} math.OA/1304.4062, 2013. To appear in \emph{J. Noncomm.
Geom.}

\bibitem{CHKL}
S.~Carpi, R.~Hillier, Y.~Kawahigashi, and R.~Longo.
\newblock {Spectral triples and the super-Virasoro algebra}.
\newblock {\em Commun. Math. Phys.} 295, 71--97, 2010.

\bibitem{CKL}
S.~Carpi, Y.~Kawahigashi, and R.~Longo.
\newblock {Structure and classification of superconformal nets}.
\newblock {\em Ann. Henri Poincare} 9, 1069--1121, 2008.

\bibitem{Con89}
A.~Connes.
\newblock {Entire cyclic cohomology of Banach algebras and characters of
$\theta$-summable Fredholm modules}.
\newblock {\em K-Theory} 1, 519–548, 1988. 

\bibitem{Con94}
A.~Connes.
\newblock {\em Noncommutative geometry}.
\newblock {Academic Press}, 1994.

\bibitem{Haag}
R.~Haag.
\newblock {\em {Local quantum physics}}.
\newblock Springer, 1992.

\bibitem{Hil1}
R.~Hillier. Super-KMS functionals for graded-local conformal nets.
\newblock {\em arXiv:} math.OA/1204.5078.

\bibitem{JLO1}
A.~Jaffe, A.~Lesniewski, and K.~Osterwalder.
\newblock {Quantum K-theory}.
\newblock {\em Commun. Math. Phys.} 118, 1--14, 1988.

\bibitem{JLW}
A.~Jaffe, A.~Lesniewski, and M.~Wisniowski.
\newblock {Deformations of super-KMS functionals}.
\newblock {\em Commun. Math. Phys.} 121, 527--540, 1989.

\bibitem{Kas}
D.~Kastler.
\newblock {Cyclic cocycles from graded KMS functionals}.
\newblock {\em Commun. Math. Phys.} 121, 345–350, 1989. 

\bibitem{KL06}
Y.~Kawahigashi and R.~Longo.
\newblock {Noncommutative spectral invariants and black hole entropy}.
\newblock {\em Commun. Math. Phys.} 257, 193–225, 2005.

\bibitem{Lo01}
R.~Longo.
\newblock {Notes for a quantum index theorem}.
\newblock {\em Commun. Math. Phys.} 222, 45–96, 2001. 

\bibitem{MS}
P.~Martin and J.~Schwinger.
\newblock {Theory of many-particle systems. I}.
\newblock {\em Phys. Rev.} 115, 1342–1373, 1959.



\end{thebibliography}
\end{document}